\documentclass[a4paper,12pt]{amsart}
\usepackage[dvips]{graphics}
\usepackage[matrix,arrow,tips,curve]{xy}
\usepackage[english]{babel}
\usepackage[leqno]{amsmath}
\usepackage{amssymb,amsthm}
\usepackage{amscd}
\usepackage{enumerate}
\usepackage{color}

\input{xy}
\xyoption{all}

\newtheorem{thm}{Theorem}[section]
\newtheorem{lemma}[thm]{Lemma}

\newtheorem{prop}[thm]{Proposition}
\newtheorem{cor}[thm]{Corollary}

\newtheorem{fact}[thm]{Fact}

\newtheorem{prob}[]{Problem}

\theoremstyle{remark}
\newtheorem{remark}[thm]{Remark}

\theoremstyle{definition}
\newtheorem{question}[thm]{Question}
\newtheorem{defi}[thm]{Definition}

\newtheorem{example}[thm]{Example}

\newcommand{\la}{\longrightarrow}
\newcommand{\ha}{\hookrightarrow}

\newcommand{\ov}{\overline}
\newcommand{\un}{\underline}

\newcommand{\im}{\operatorname{Im}}

\newcommand{\Div}{\operatorname{Div}}

\newcommand{\supp}{\operatorname{Supp}}
\newcommand{\Spec}{\operatorname{Spec}}

\newcommand{\Pic}{\operatorname{Pic}}

\newcommand{\Prin}{\operatorname{Prin}}

\newcommand{\Picphi}{\operatorname{Pic}_{\phi}}

\newcommand{\mdeg}{\underline{\operatorname{deg}}}

\def\L{\mathcal L}
\def\O{\mathcal O}

\def\X{\mathcal X}

\newcommand{\PP}{\mathbb{P}}
\newcommand{\Z}{\mathbb{Z}}

\def\md{\underline{d}}

\def\me{\underline{e}}

\def\mk{\underline{k}}
\def\mt{\underline{t}}

\def\mc{\underline{c}}

\newcommand{\MaG}{M^{\rm alg}(G)}
\newcommand{\Mgb}{\ov{M_g}}

\newcommand{\rmax}{r^{\rm{max}}}
\newcommand{\RX}{\rmax(X,\md)}

\newcommand{\rdel}{r(X,\delta)}
\newcommand{\ralg}{r^{\rm{alg}}}
\newcommand{\rgdel}{\ralg(G,\delta)}
 
\newcommand{\rGdel}{r_G(\delta)}

\newcommand{\hhG}{\widehat{G}}
\newcommand{\hhd}{\widehat{\md}}

\newcommand{\w}{\omega}
\newcommand{\meg}{\me^{\deg}}
\newcommand{\mdr}{\md_{{\operatorname{rk}}}}
\newcommand{\agd}{\ell_G (\md)}

\begin{document}
\title[Algebraic and combinatorial rank of divisors on graphs]{Algebraic and combinatorial rank of divisors on finite graphs}
  \author{Lucia Caporaso, Yoav Len, and Margarida Melo}
 \maketitle

\begin{abstract}
We study the algebraic rank  of a divisor on a  graph, an invariant   defined 
using  divisors on  algebraic curves dual to the graph. We prove it satisfies  the Riemann-Roch formula, 
a specialization property, and the Clifford inequality. We prove that  it is at most equal to  the (usual) combinatorial rank,
and that equality holds in many  cases, though not in general.
\end{abstract}
 \tableofcontents

\section{Introduction}

Since recent years, a   lively   trend of research is  studying the
interplay between   the combinatorial and   algebro-geometric aspects of the theory of algebraic curves;
this has led to interesting progress both in algebraic geometry and graph theory. The goal of this paper is to contribute to this progress by investigating the connection between the notion of combinatorial rank of divisors on graphs, and the notion of rank of      Cartier divisors on an algebraic curve. 

Loosely speaking, our main result is that the combinatorial rank of a divisor on a graph
(a computer-computable quantity  bounded above by the  degree)  is a fitting uniform
upper bound on the dimension of linear series on curves (a hard to compute quantity, unbounded regardless of the  degree).
To be  more precise, we   need   some context.

 In the theory of algebraic curves, combinatorial aspects naturally appear when dealing with all curves simultaneously, as points of an algebraic variety.
Indeed,
a typical phenomenon in algebraic geometry is   that the set of equivalence classes of  varieties with given discrete invariants
 is itself an algebraic variety, whose geometric properties reflect those of the varieties it parametrizes.
 The case we should here keep  in mind is the space $\Mgb$, of all  connected  nodal curves of genus $g$ up to stable equivalence; it  is a  complete variety  containing, as a dense open subset, the space of isomorphism classes of smooth curves.
$\Mgb$ has been a central object of study for a long time, and it has been successfully used to  study
the geometry of algebraic curves. Several   topics in this field, among which many open problems,
concern projective realizations of abstract curves, i.e.  the theory of line bundles  (or  Cartier divisors)  and linear series.

 A systematic study of these matters requires    combinatorial methods to handle singular curves.
Moreover,  several   questions  are successfully answered by  degeneration techniques  (specializing a smooth curve to a singular one)
where combinatorial aspects are essential;
examples of this are the Griffiths-Harris proof of the Brill-Noether theorem, in \cite{GH}, or the Kontsevich,  and others,  recursive formulas  enumerating curves on surfaces, see \cite{KM} or \cite{CH}.

 In fact, since the first appearances of $\Mgb$, as in 
 the seminal paper
 \cite{DM}, one sees  associated to every nodal curve its {\it dual graph},
having as vertices the irreducible components of the curve, and as edges the nodes of the curve; moreover, every vertex of the dual graph is given a {\it weight}, equal to the geometric genus of the component it represents.
For any  (weighted) graph $G$ we denote by $\MaG$
the set of isomorphism classes of curves dual to $G$ (i.e. having $G$ as dual graph). Then we have
$$
\Mgb =(\sqcup \MaG)/\sim
$$
where the union is over all  connected graphs of genus $g$, and ``$\sim$" denotes stable equivalence
(which we don't  define here, see \cite{HMo}).

The dual   graph is a key tool to deal with the combinatorial aspects mentioned above, especially
in the theory of
divisors and line bundles, when studying   N\'eron models of Jacobians,  Picard functors and compactified Jacobians, or degenerations of linear series.

More recently, and independently of the picture we just described,     a purely combinatorial  theory of divisors  and linear series on  graphs was being developed   in a     different framework; see \cite{BdlHN} and \cite{BNRR}.
The discovery that this graph-theoretic  theory  fits  in  well with the algebro-geometric set-up
 came somewhat as a surprise.
To begin with, the group, $\Div(G)$,  of divisors  on a graph $G$ is the free abelian group on its vertices.
The  connection to line bundles on curves is simple:   given a curve   $X$   dual to $G$,
the multidegree of a line bundle   on $X$
is naturally a divisor on $G$, so we have a map  $\Pic(X)\to \Div(G)$.

Such  developments in graph-theory provide   a  fertile ground to extract and study the combinatorial aspects of the theory of algebraic curves;  a remarkable example of this is the recent proof of the above mentioned Brill-Noether theorem, given in \cite{CDPR}.

In this spirit, as mentioned at the outset, the goal of this paper is to interpret  the  combinatorial  rank, $r_G(\md)$,  of a divisor, $\md$,  on a graph $G$
(as defined in \cite{BNRR} and in \cite{AC})
by  the theory of algebraic curves. We do that  by studying another invariant,
the  {\it algebraic rank}
$$
\ralg (G,\md)
$$ of the divisor $\md$ ,
defined in a completely different fashion, using   algebro-geometric notions.   

 In algebraic geometry,  the notion  corresponding to the  combinatorial rank  is the
rank of a line bundle, i.e. the dimension of the space of its global sections diminished by one.
Now,  the algebraic rank of a divisor on  a graph $G$ should be thought of as a uniform  ``sensible"  upper bound
on the rank of any line bundle  on any curve in $\MaG$ having  multidegree equal to the given divisor on $G$.
The word ``sensible" signals the fact that we need the algebraic rank to be constant in   equivalence classes, hence the precise definition needs some  care
 (see section~\ref{ssalgrank} for details).
However, the following simple  case is clear and gives the right idea: let  $G$ be just one vertex of weight $g$
(no edges),
  then $\MaG$ parametrizes smooth curves of genus $g$,
and a divisor on $G$ is  an integer, $d$; if  $d$ is non negative and at most $2g-2$, the  theorem of Clifford  yields  that the algebraic rank of $d$ is equal to
$\lfloor d/2\rfloor$, and the bound  is achieved exactly by certain line bundles on hyperelliptic curves.
 
 The algebraic rank was introduced in \cite{Cjoe}, where it was conjectured to coincide with the combinatorial rank.
Although, as we will here prove,  this conjecture is true in a large number of cases,   it fails in general.
For convenience, we will now assemble together our results concerning the relation between combinatorial and algebraic rank.

\noindent
{\bf{Comparison  of combinatorial and algebraic rank: results.}}
{\it  Let $G$ be a  (finite, connected, weighted) graph of genus $g$, and let $\md$ be a divisor on $G$. Then 
 we have 
\begin{equation}
 \label{mainresult}
 \ralg (G,\md) \leq r_G(\md).\end{equation}
 Moreover, equality holds in the
 following cases.
\begin{enumerate}[(a)]
\item $r_G(\md)\leq 0$;
\item
\label{22}
$G$ is a binary graph (see Example~\ref{spex});
\item
$G$ is loopless, weightless and $\md$ is rank-explicit (see Definition~\ref{rex}).
\end{enumerate}}
Some recent results of \cite{KY2} establish the inequality opposite to  \eqref{mainresult} in some cases; hence, combining with our results, we have that  equality holds in \eqref{mainresult} for   non-hyperelliptic graphs if $g=3$, and for all hyperelliptic graphs in characteristic other than $2$; see Corollary~\ref{KY}. For completeness, we mention that equality in \eqref{mainresult}   holds
 when  $g\leq 2$ (assuming $G$ stable for $g=2$), and if    $d\geq 2g-2$ or $d\leq 0$; see \cite{Cjoe}.

As we said, we have some  counterexamples 
showing that   equality can  fail  in \eqref{mainresult}; namely
Example~\ref{wbinex} shows that the hypothesis that $G$ is weightless cannot be removed from \eqref{22}, and
Example~\ref{counterex1} shows that  equality can fail  for weightless graphs with three vertices.
We do not know how to characterize cases where we have a strict inequality in \eqref{mainresult}, but we believe that
it would be quite an interesting problem.
  
Here is an outline of  the paper. We start with a  study  of the  algebraic rank,  establishing  some basic properties:  in Section~\ref{arsec} we prove that it satisfies the Riemann-Roch formula,
 along with the appropriate version of Baker's specialization lemma; both results are proved using algebraic geometry, with no relation to the analogous facts for the combinatorial rank.

In Section~\ref{crsec} we concentrate on  the combinatorial rank   and establish ways to compute it, or  to bound it from above.
We use the theory of reduced divisors
from \cite{BNRR}, implementing it  by constructing
what we call the {\it Dhar  decomposition} of the vertex set of a graph, a  particularly useful tool for our goals.

Using the material of the previous parts,  in Section~\ref{insec} we
 prove that
 the algebraic rank is always at most equal to   the combinatorial rank (Theorem~\ref{rankthm}).
 As a   consequence, the algebraic rank   satisfies the Clifford inequality; this fact could not  be proved using algebro-geometric methods (like the Riemann-Roch formula above), as the Clifford inequality   fails for reducible curves.

The last section focusses on the  opposite inequality: when is the algebraic rank  at least equal to
the combinatorial rank? We now must  bound the algebraic rank from below, hence the new issue   is to find singular curves which play for us the role of hyperelliptic curves (as in the above example),
  i.e. curves on which line bundles of a fixed multidegree tend to have the highest possible rank.
  We find     such {\it special curves} in every set $\MaG$ when $G$ is  a weightless graph,     and  we use them to prove  that, for   so-called  {\it rank explicit}   divisors on weightless graphs, the combinatorial rank and the algebraic rank are equal (Theorem~\ref{rexthm}).

\

\noindent
 {\it Acknowledgments.}  We wish to thank the referee for helpful comments, and Sam Payne for introducing the second   author
to the first,  thus making  our collaboration   possible.

 The third author was partially supported by CMUC, funded by the European Regional Development Fund through the program COMPETE and by FCT under the project
PEst-C/MAT/UI0324/2013 and grants PTDC/MAT/111332/2009, PTDC/MAT-GEO/0675/2012 and SFRH/BPD/90323/2012.

\begin{table}[htbp]
\begin{center}
{\sc Index of non standard notations}

\

\begin{tabular}{l p{12cm}  }
$\md$ & divisor on a graph $G$.\\
$\me$ & effective divisor on $G$.\\
$\delta$ & divisor class on $G$, usually the class of $\md$.\\
$\mt_Z$ & principal divisor corresponding to a set of vertices $Z$; see \eqref{tZdef}.\\
$|\md |$ & degree of   $\md$.\\
$r_G(*)$ & combinatorial rank of $*$.\\
$\ralg(G,*)$ & algebraic rank of $*$; see Def.~\ref{ardef}.\\
$\widehat{G}$ & loopless weightless graph associated to   $G$; see Sec.~\ref{wsec}.\\
$g(v)$ &  sum of weight and number of loops at $v$; see Sec.~\ref{wsec}.\\
$R_v$ &  subgraph of $\widehat{G}$ consisting of $v$ and the $g(v)$ attached cycles; see \eqref{Rv}.\\
$e^g$ & $e+\min\{e,g\}$, where $e$ and $g$ are non-negative integers.\\
$\me^{\deg}$ & divisor on $G$ satisfying $\me^{\deg}(v) = \me(v)^{g(v)}$; see Def~\ref{defmeg}.\\ 
$d_g$ & $\max\{d-g,\lfloor\frac{d}{2}\rfloor\}$, where $d$ and $g$ are non-negative integers.\\
$\md_{\rm{rk}}$ & divisor on $G$ satisfying $\md_{\rm{rk}}(v) = \md(v)_{g(v)}$; see Def~\ref{defmdr}.\\
$\ell_G(\md)$ &  See \eqref{ad}  for $G$ weightless, and \eqref{adw} in general.  \\
\end{tabular}

\end{center}
\end{table}

 \section{Algebraic rank of divisors on graphs}
 \label{arsec}
\subsection{Preliminaries}
Throughout the  paper  $G$
will denote
 a (vertex weighted) finite graph, $V(G)$ the  set of its vertices, $E(G)$ the  set of its  edges and $\Div(G)=\Z^{V(G)}$ the group of its divisors; when no ambiguity is likely to occur we write $V=V(G)$ and $E=E(G)$.
The weight function of $G$ is written $\omega:V\to \Z_{\geq 0}$; if $\w(v)=0$ for every $v\in V$ we say that $G$ is weightless.

  The {\it genus} of  a connected graph  $G$ is
$$g(G):=\sum_{v\in V}\omega(v)+ |E|-|V|+ 1;$$
thus, if $G$ is weightless, $g(G)$ is its first Betti number.
If $G$ has $c$ connected components, $G_1,\ldots, G_c$,   we set $g(G)=\sum_1^c g(G_i)+1-c$.
Our graphs   will always  be connected,
unless otherwise specified.

Elements in $\Div(G)$ are usually denoted by underlined lower-case letters, for example we write $\md\in \Div(G)$ and  $\md=
\{\md(v),\  \forall  v\in V\}$ with $\md(v)\in \Z$. We write $|\md|:=\sum_{v\in V}\md(v)$ for the ``degree" of a divisor $\md$.

We write $\md\geq 0$ for an effective divisor (i.e. such that $\md(v)\geq 0$ for every $v\in V$); if $\md$ and $\me$ are effective,
and $\md-\me$ is also effective, we say that  ``$\md$ contains $\me$".

We will usually abuse notation  so that for a  set $Z\subset V$ of vertices, we also denote by $Z$
the divisor $\sum_{v\in Z}v\in \Div(G)$ (or, with the above notation, the divisor $\md$ such that $\md(v)=1$ for all $v\in Z$, and $\md(v)=0$
otherwise).

The group $\Div(G)$ is endowed with an intersection product associating to $\md_1,\md_2\in \Div(G)$ an integer, written $\md_1\cdot \md_2$.
If $\md_i=v_i$ with $v_i\in V$ and $v_1\neq v_2$ we set  $v_1\cdot v_2$ equal to the number of edges joining $v_1$ with $v_2$, whereas
$v_1\cdot v_1=-\sum _{v\in V\smallsetminus \{v_1\}}v\cdot v_1$.
The linear extension to  the entire $\Div(G)$ gives our intersection product.

The subgroup $\Prin (G)\subset \Div (G)$ of  {\it principal}  divisors is generated by divisors of the form
$\mt_Z$, for all $Z\subset V$; these principal divisors are defined so that
for any $v\in V$ we have

\begin{equation}
\label{tZdef}
\mt_Z(v) =\left\{ \begin{array}{ll}
v\cdot Z  \  &\text{ if }   v\not\in Z\\
-v\cdot Z^c \  &\text{ if }   v\in Z\\
\end{array}\right.
\end{equation}
where $Z^c:=V\smallsetminus Z$.
\begin{remark} \label{tZrk}
 Let $\mt\in \Prin(G)$ be non-trivial. Then there exists a partition $V=Z_0\sqcup\ldots \sqcup Z_m$,
 with $Z_0$ and $Z_m$ non-empty,
 such that $\mt=\sum_{i=0}^mi\mt_{Z_i}$.
Indeed,   by definition, 
 $\mt=\mt_{Y_1}+...+t_{Y_k},$  where each $Y_j$ is a set of vertices. For each $a\geq 0$, let $Y'_a$ be the set of vertices that are contained in $a$ different such sets ($Y'_a$ may be empty). Then the sets $Y'_a$ are a disjoint cover of $V$, and $\mt = 0\cdot \mt_{Y'_0}+\ldots+k\cdot\mt_{Y'_k}$.
Let $b$ be the first integer so that $Y'_b$ is non-empty. Since $\sum \mt_{Y'_i}=\mt_V=0$, by subtracting $b$ copies of it from $\mt$, and defining $Z_i = Y'_{i+b}$, we are done. 
 
  This implies that  we have
  $$
 \mt_{|Z_m}\leq (\mt_{Z_m})_{|Z_m}.
 $$
Indeed, pick $v\in Z_m$, we have $ \mt_{Z_m}(v)=-Z_m^c\cdot v$; on the other hand
 $$
 \mt(v)=\sum_{i=0}^{m-1}iZ_i\cdot v - mZ_m^c\cdot v\leq (m-1)\sum_{i=0}^{m-1}Z_i\cdot v -mZ_m^c\cdot v=-Z_m^c\cdot v.
 $$
\end{remark}

We set $\Pic(G)= \Div (G)/\Prin (G)$; we say that two divisors are   {\it equivalent} if their difference is in
$\Prin(G)$, and we refer to $\Pic(G)$ as
the group of  equivalence classes of divisors; for
 $\md\in \Div (G)$ we denote its class in $\Pic (G)$ by $\delta=[\md]$. Since   equivalent divisors have the same degree
 we write $|\delta|:=|\md|$.
 For an integer $k$ we write $\Div^k(G)$, and $\Pic^k(G)$, for the set of divisors, or divisor classes, of degree $k$.

For any  divisor $\md\in \Div(G)$ we have its {\it combinatorial rank}, denoted by $r_G(\md)$ and defined
as in \cite{BNRR} and \cite{AC} (see subsection~\ref{wsec} for details). For now, we just recall that
$r_G(\md)$ is an integer such that
$$
-1\leq r_G(\md) \leq \max\{-1, |\md|\}.
$$
Equivalent divisors have the same  combinatorial rank, hence we set
$$
r_{G}(\delta):=r_G(\md).
$$

 \subsection{The algebraic rank}
 \label{ssalgrank}
In this paper, unless otherwise specified, the word {\it curve} stands for   one-dimensional scheme,   reduced and projective over an algebraically closed field $k$,
and having only nodal singularities. Let $X$ be a curve; its {\it dual graph}, $G$,
is   defined as follows. $V(G)$ is  the set of   irreducible components of $X$;
$E(G)$ is the set of   nodes of $X$, with one edge joining two vertices if the corresponding node is at the intersection of the two corresponding components (a loop based at a vertex $v$ corresponds to a node of the irreducible component   corresponding to $v$);
 the weight  of $v\in V$ is   the genus of the desingularization of the corresponding component.
One easily checks that the (arithmetic) genus of $X$ is equal to the genus of its dual graph.
We recall that it is well known that the theorem of Riemann-Roch holds for any such curve $X$, but the Clifford inequality only holds  when  $X$ is irreducible; see \cite {Ccliff} for more details.

The irreducible component decomposition of $X$ is
denoted as follows
$$
X=\cup_{v\in V}C_v.
$$

$\Pic (X)$ is the group of isomorphism classes of line bundles on $X$; for $L\in \Pic(X)$
we denote by
$\mdeg L=\{\deg_{C_v}L,\  \forall v\in V\}$
its  {\it multidegree}. Now, $\mdeg L$ can be viewed as a divisor on $G$, by setting
$\mdeg L(v)=\deg_{C_v}L$;
 hence we have a surjective group homomorphism
$$
\Pic(X)\la \Div(G); \quad \quad L \mapsto \mdeg L.
$$
For $\md \in\Div(G)$ we write $\Pic^{\md}(X)=\{L\in \Pic(X):\  \mdeg L = \md\}$.

As already mentioned, $\MaG$ denotes the set of  isomorphism classes of curves dual to $G$.
Observe that $\MaG$ is never empty.

For any  $X\in \MaG$ and  any $\md \in \Div(G)$ we define
$$
\RX:=\max \{r(X,L),\  \    \forall L\in \Pic^{\md}(X)\}
$$
where $r(X,L)=\dim H^0(X,L)-1$. Note the following simple   fact:
\begin{remark}
\label{monotone}
Let  $\md'\in \Div (G)$; if $\md'\geq \md$ then $\rmax(X,\md ')\geq \RX$.
\end{remark}
Varying $\md$ in its equivalence class  $\delta\in \Pic (G)$ we can define
$$
\rdel:=\min \{\RX,\  \  \  \forall \md\in \delta\}.
$$
(By contrast,  $\max \{\RX,\   \forall \md\in \delta\}=+\infty$, for every reducible   curve $X$).
\begin{defi}
\label{ardef}
For any divisor class $\delta\in \Pic(G)$ of a graph $G$ we set
$$
\rgdel:=\max \{\rdel,\  \  \  \forall X\in \MaG\},
$$
and for every representative $\md\in \Div(G)$ for $\delta$
$$
\ralg (G,\md):=\rgdel.
$$
We refer to $\rgdel$ and   $\ralg (G,\md)$  as the {\it algebraic rank} of $\delta$ and $\md$.
 \end{defi}
\begin{example}\label{irrex}
 If $G$ has only one vertex,  any curve $X\in \MaG$ is irreducible 
 (singular if $G$ has loops); then
  every  class  in $\Pic(G)$ contains exactly one element, and  we naturally identify $\Pic(G)=\Div(G)=\Z$.
 Then, for any $d\in \Div (G)$
  by the theorems of Riemann-Roch and Clifford,
  \begin{displaymath}
\ralg(G,d)=\left\{ \begin{array}{ll}
-1  \  &\text{ if }   d<0\\
\lfloor d/2\rfloor \  &\text{ if }   0\leq d\leq 2g-2\\
d-g \  &\text{ if }    d\geq 2g-1.\\
\end{array}\right.
\end{displaymath}
\end{example}
The following natural problem arises
\begin{prob}
\label{toy2}
Let $G$ be a graph  and $\delta\in \Pic (G)$.
Is $$
\rgdel=r_{G}(\delta)?
$$
\end{prob}
In \cite{Cjoe}  the answer to this problem is shown to be positive in a series of  cases, and it is conjecture that it be always the case.
   As already  mentioned, we
  shall prove   that  we have  $\rgdel\leq r_{G}(\delta)$, but equality does fail in some cases.

\begin{remark}
 In this paper we do not consider metric graphs. Nonetheless, we believe   it would be very interesting to  have a generalization  of the  algebraic rank to divisors on metric graphs in a way that reflects the algebro-geometric nature of the graph, which we can loosely describe as follows.
To a metric graph $\Gamma=(G,\ell)$ there corresponds
 a set of   (equivalence classes of) nodal curves, $\X\to \Spec R$, where  $R$ is a valuation ring with algebraically closed residue field; the dual graph of the closed fiber   is $G$, and the metric, $\ell$, represents the  geometry of $\X$ locally at the closed fiber.
 
 We refer to  
  \cite{AB} for a   treatment of the correspondence between   algebraic and combinatorial aspects  of the theory using metric graphs
(and more generally,  ``metrized complexes") instead of finite graphs.

  \end{remark}
 
\subsection{Riemann-Roch for  the algebraic rank}\label{S:properties}
We shall now  prove that the algebraic rank, exactly as the combinatorial rank,
 satisfies a Riemann-Roch formula; the proof  will be   a  consequence of  Riemann-Roch for curves.

As we said, the Clifford inequality also holds for the algebraic rank,   but its proof requires more  work  and it is quite different as  it follows from the Clifford inequality for graphs (indeed, the Clifford inequality  fails for reducible  curves!); 
see Proposition~\ref{cliffcor}.

We denote by $\mk_G$ the canonical divisor of a graph $G$, and by $K_X$ the dualizing line bundle of a curve $X$.
Recall that $\mk_G$ is defined as follows
$$
\mk_G(v) = \text{val}(v) + 2\w(v) - 2,
$$
where $\text{val}(v)$ denotes the valency of the vertex $v$.
 We have $\mdeg K_X=\mk_G$ for every  $X\in \MaG$.
\begin{prop}[Riemann-Roch]
\label{RRalg}
Let $G$ be a graph of genus $g$, $\md$ a divisor of degree $d$ on $G$, and $X\in \MaG$. 
Then, setting  $\delta=[\md]$, the following identities hold.

\begin{enumerate}[(a)]
 \item
\label{RRX}
$\rmax(X,\md)-\rmax(X,\mk_G-\md)=d-g+1$;\\
 \item
\label{RRdelta}
$r(X,\delta)-r(X,\mk_G-\delta)=d-g+1$;\\

 \item
 \label{RRG}
$\ralg(G,[\md])-\ralg(G,[\mk_G-\md])=d-g+1.$

\end{enumerate}
\end{prop}

\begin{proof}
We begin by introducing some notation.
For $L\in \Pic^{\md}(X)$, set $L^*=K_XL^{-1}$, so that $L^{**}=L.$ 
Similarly, we set $\md^*=\mdeg L^*=\mk_G-\md$ and  $d^*=\deg\md^*=2g-2-d$, next  $\delta^*:=[\md^*]$
(this is well defined, as  $\md\sim \me$  implies $\md^*\sim \me^*$ ).  We have $\delta^*\in \Pic^{\md^*}(G)$ and $\delta^{**}=\delta$ (as $\md^{**}=\md$).  

Notice that the correspondence $L\mapsto L^*$ is a bijection between $\Pic^{\md}(X)$ and $\Pic^{\md^*}(X)$. Similarly $\md\mapsto \md^*$ is a bijection between the representatives of $\delta$ and those of $\delta^*$, and   $\delta\mapsto \delta^*$ is a bijection   between  $\Pic^{d}(G)$ and $ \Pic^{d^*}(G)$.

Let $X\in \MaG$ and $L\in \Pic^{\md}(X)$. We claim  the following:
\begin{equation}
\label{rmax*}
\rmax(X,\md)=r(X,L) \  \  \Leftrightarrow \  \  \rmax(X,\md^*)=r(X,L^*).
\end{equation}
In other words, whenever $L$ is a line bundle realizing $\rmax(X,\md)$, its dual $L^*$ will   realize $\rmax(X,\md^*)$.
By the algebro-geometric Riemann-Roch  applied to $L$ on $X$ it is clear that \eqref{rmax*} implies \eqref{RRX}.

By the bijection described above, it suffices to prove only  one implication of  \eqref{rmax*}.
So assume $\rmax(X,\md)=r(X,L)$. By contradiction, suppose $r(X,L^*)<\rmax(X,\md^*)$, and let $M^*\in \Pic^{\md^*}(X)$ be such that $r(X,M^*)=\rmax(X,\md^*)$.
Now by Riemann-Roch for $X$ we have
$$
r(X,L)=r(X,L^*)+d-g+1<r(X,M^*)+d-g+1=r(X,M)
$$
hence $\rmax(X,\md)=r(X,L)<r(X,M)$, which is impossible as $M\in \Pic^{\md}(X)$.
\eqref{rmax*} is thus proved, and  \eqref{RRX} with it.

From   \eqref{RRX}, to prove \eqref{RRdelta} it suffices to show the  following:
\begin{equation}
\label{rmax*2}\rmax(X,\md)=r(X,\delta)\  \  \Leftrightarrow \  \   \rmax(X,\md^*)=r(X,\delta^*).
\end{equation}
As before, we need only prove one implication, so assume $
\rmax(X,\md)=r(X,\delta)$  and let $L\in \Pic^{\md}(X)$ be such that
$r(X,L)=\rmax(X,\md)$. By \eqref{rmax*} we have
$r(X,L^*)=\rmax(X,\md^*)$, so it suffices to
  prove that
$r(X,L^*)=r(X,\delta^*)$.
By contradiction, suppose this is not the case. Then there exists $\me^*\in \delta^*$ and $N^*\in \Pic^{\me^*}(X)$ such that
$$
r(X,L^*)>r(X,N^*)=\rmax(X,\me^*).
$$
By Riemann-Roch on $X$ we have
$$
r(X,L)=r(X,L^*)+d-g+1>r(X,N^*)+d-g+1=r(X,N).
$$
By \eqref{rmax*}   we have
$$r(X,N)=\rmax(X,\me)\geq \rmax(X,\md)=r(X,L),$$
a contradiction with the previous inequality; \eqref{rmax*2} and \eqref{RRdelta} are proved.

Finally,   let  $L\in \Pic^{\md}(X)$   be such that
$r(X,L)=\rgdel$, and let us prove that $r(X,L^*)=\ralg(G,\delta^*)$.
By Riemann-Roch on $X$ this  will imply  \eqref{RRG}.

As $r(X,L)=\rmax(X,\md)=r(X,\delta)$   by  \eqref{rmax*} and \eqref{rmax*2} we have
$r(X,L^*)=\rmax(X,\md^*)= r(X,\delta^*)$. By contradiction, suppose there exists a curve  $Y\in \MaG$ and a line bundle $P^*\in \Pic^{\me^*}(Y)$ with $\me^*\in \delta^*$ such that
$$
r(X,L^*)<\ralg(G,\delta^*) = r(Y,P^*)=\rmax(Y, \me^*).
$$
 Arguing as before we get
$$
r(X,L)=r(X,L^*)+d-g+1<r(Y,P^*)+d-g+1=r(Y,P).
$$
Now claims \eqref{rmax*} and \eqref{rmax*2}  yield, as $\me\in \delta$,
$$r(Y,P)=
r(Y,\delta)\leq
\rgdel=r(X,L),$$
 contradicting  the preceding inequality.
The proof is complete.
\end{proof}
   \subsection{Specialization for  the algebraic rank}
We shall now prove a result analogous to    Baker  Specialization Lemma, established in \cite{bakersp},  stating that the algebraic rank of divisors varying in a family of curves cannot decrease when specializing to the dual graph of the special fiber.

We need some preliminaries.
Let $X$ be a connected  curve, and let
  $\phi:\X \to (B,b_0)$ be a regular one-parameter
smoothing of    $X$, i.e.
$\X$ is  a regular  2-dimensional variety, $B$ is a regular  1-dimensional variety
with a marked point $b_0\in B$,
 $\phi$ is a fibration in curves such that the fiber over $b_0$ is  $X$ and the fibers over the other points of $B$ are smooth projective curves.
The relative Picard scheme of $\phi$ is written $\Picphi\to B$, so that the fiber of  $\Picphi$ over a point $b\in B$ is the Picard scheme
of the curve $X_b:=\phi^{-1}(b)$. The set of sections of $\Picphi\to B$  is denoted by $\Picphi(B)$;
so, an element $\L\in \Picphi(B)$ gives a line bundle $\L (b)\in \Pic(X_b)$ for every $b\in B$.

Let $L_0$ and $L'_0$ be two line bundles on $X$; we say that $L_0$ and  $L'_0$ are $\phi$-equivalent, and write
$L'_0\sim_{\phi} L_0$, if for some divisor $D$ on $\X$ entirely supported on $X$ we have
$$
L_0^{-1}\otimes L'_0= \O_\X(D)_{|X}.
$$
 For example, for $\L\in \Picphi(B)$  and $D\in \Div(X)$ as above, we can define $\L'\in \Picphi(B)$
 that assigns $\L(b)\otimes \O_\X(D)_{|X_b}$ to every $b\in B$. As  $\supp D\subset X$ we have
 $\L(b)=  \L'(b)$ for $b\neq b_0$, and $\L(b_0)\sim_{\phi}\L'(b_0)$.

 Finally, let $G$ be the dual graph of  $X$; for any $\phi:\X\to B$ as above we have a natural map (cf. subsection~\ref{ssalgrank})
 $$
\tau: \Picphi(B)\la \Div(G); \quad \quad \L\mapsto \mdeg \ \L(b_0).
 $$

\begin{lemma}[Specialization]
 \label{spe}
Let $\phi:\X \to B$ be a regular one-parameter
smoothing of  a connected  curve $X$.
Let $G$ be the  dual graph of $X$.
Then for every $\L\in \Picphi(B)$ there exists an open neighborhood $U\subset B$ of $b_0$ such that for every $b\in U\smallsetminus\{b_0\}$   we have
\begin{equation}
\label{mixed}
r(X_b, \L (b))\leq \ralg(G, \mdeg \ \L(b_0)).
\end{equation}
\end{lemma}

\begin{proof}
We write $L_b: =\L (b)$,\  $L_0: =\L (b_0)$  and $\md:=\tau (\L)=\mdeg L_0$; we set $\delta\in \Pic (G)$ to be the class of $\md$.
By uppersemicontinuity of $h^0$ we have, for every $L'_0\in \Pic(X)$ such that $L'_0\sim_{\phi} L_0$
$$
r(X_b,L_b)\leq r(X, L'_0)
$$
for every $b$ in some neighborhood $U\subset B$ of $b_0$.
Hence, by definition of $\rmax$,
$$
r(X_b,L_b)\leq \rmax(X, \mdeg L'_0).
$$
As $L'_0$ varies in its $\phi$-class  the values of   $\mdeg  L'_0$ cover all the representatives of $\delta$,  therefore
we obtain
$$
r(X_b,L_b)\leq r(X, \delta).
$$
Since by definition $r(X, \delta)\leq \ralg(G, \delta)$, we are done.
\end{proof}
\subsection{Algebraic smoothability of divisors on graphs}
Let us now show how   Problem~\ref{toy2} is related to  the ``smoothability" problem for line bundles on curves, i.e.  the following general natural problem.
\begin{prob}
\label{pb1}
Let   $L$ be a line bundle on a curve $X$  such that  $r(X,L)=r$. Do  there exist a regular one-parameter smoothing $\phi:\X \to B$
 of $X$, and a section $\L\in \Pic_{\phi}(B)$, such that  $ \L(b_0)=L$ and $r(X_b, \L(b))= r$ for every $b\in B$?
\end{prob}
Easy cases when the answer is always   positive are $r<0$ (by upper-semicontinuity) and $\deg L\geq 2g-1$ (by Riemann-Roch).
But in other cases   this   problem is well known to get   very hard, even without the regularity assumption on $\phi$. We shall now try to simplify it  by   a combinatorial version.
\begin{defi}
 Let $G$ be a graph and $\md\in \Div(G)$; set $r_G(\md )=r$.
 We say that $\md$ is {\it (algebraically) smoothable} if there exist  $X\in \MaG$, a regular one-parameter smoothing $\phi:\X \to B$
 of $X$, and a section $\L\in \Pic_{\phi}(B)$,  such that  $\mdeg \ \L(b_0)=\md$ and $r(X_b, \L(b))= r$ for every $b\in B\smallsetminus b_0$
 (hence $r(X, \L(b_0))\geq r$ by upper-semicontinuity).
\end{defi}

   The following is the   combinatorial counterpart to  Problem~\ref{pb1}.
   \begin{prob}
\label{pb2}
 Let $\md\in \Div(G)$. Is $\md $ smoothable?
\end{prob}
As for  Problem~\ref{pb1},  if $r_G(\md)=-1$, or if $|\md|\geq 2g(G)-1$, then  $\md$   is smoothable

\begin{remark}
 If $\md$ is smoothable and $\md'\sim \md$, then $\md'$ is also smoothable
(by the regularity assumption on $\phi$). We shall say that a divisor class $\delta\in \Pic(G)$ is smoothable if so are its representatives.
\end{remark}

The next simple result connects  Problem~\ref{pb2} to Problem~\ref{toy2}.
\begin{prop}
 If $r_G(\delta)>\ralg (G,\delta)$ then $\delta$ is not smoothable
 (i.e. no representative for  $\delta$ is smoothable). \end{prop}
\begin{proof}
 Set $r=r_G(\delta)$. By contradiction, assume that $\delta$ is smoothable.
 Pick any representative $\md$ for $\delta$;
 let $\phi:\X\to B$ and $\L$ as in the above definition, with $\md = \mdeg\ \L(b_0)$.
 Then
 $$
 r\leq r(X, \L(b_0)) \leq  \rmax(X,\mdeg\ \L(b_0))= \rmax(X,\md).
 $$
 Since the above holds for every $\md$ in $\delta$, we get  $r(X,\delta)\geq r$. Therefore
 $$
 r \leq r(X,\delta)\leq \ralg(G,\delta),
 $$
 a contradiction.
\end{proof}
As we shall see in the examples at the end of Section~\ref{spesec}, non-smoothable classes do exist. We refer to the recent preprint \cite{Ca} for a study of closely related issues and further examples.

We  conclude this section asking  about   the converse of the above proposition, namely, what can be said about the smoothing problem when the algebraic and combinatorial rank coincide
(keeping in mind that by Theorem~\ref{rankthm} below, the algebraic rank is never greater than the combinatorial rank).

\begin{question}
Assume $r_G(\delta)=\ralg (G,\delta)$; is   $\delta$ smoothable?
\end{question}
A negative answer to this question is given in \cite[Example 3.2]{Yoav2}, where a non-smoothable divisor whose algebraic and combinatorial ranks are both $1$ is found.
\section{Computing   combinatorial ranks via reduced divisors}
\label{crsec}
\subsection{Combinatorial rank of divisors on   graphs}
\label{wsec}
Let $G$ be a graph. For each $v\in V(G)$ we denote by   $l (v)$ the number of loops adjacent to $v$; we  set
$$g(v):=\w(v)+l(v).$$
Let $\hhG$ be the weightless, loopless graph obtained from $G$ by gluing  to each vertex $v\in V(G)$ a number of loops equal to $\w(v)$, and then by inserting a vertex in every loop.
Denote by $z_v^1,\dots,z_v^{g(v)}$ the vertices in $V(\hhG)\smallsetminus V(G)$ adjacent to $v$, and  by $R_v$ the complete subgraph of $\hhG$ whose vertices are $\{v,z_v^1,\dots,z_v^{g(v)}\}$:
\begin{equation}
 \label{Rv}
 R_v:=\big[v,z_v^1,\dots,z_v^{g(v)}\big]\subset \hhG ;
\end{equation}
note that $R_v$ has genus $g(v)$.

There is a natural injective homomorphism
\begin{equation}
\label{Divmap}
\Div(G)\ha \Div(\hhG); \quad \quad \md\mapsto \hhd
\end{equation}
such that $\hhd$ is defined to be zero  on  each new vertex of $\hhG$, and equal to $\md$ on the vertices of $G$.
The above  homomorphism maps $\Prin(G)$ to $\Prin(\hhG)$,
hence it  descends to an injective homomorphism
$$ \Pic(G)\ha \Pic(\hhG); \quad \quad \delta\mapsto \widehat{\delta}$$
such that if $\delta=[\md]$ then $\widehat{\delta}=[\hhd]$.
The map \eqref{Divmap}  is used to define the combinatorial rank of a divisor on a general graph  $G$ via the combinatorial rank of a divisor on the weightless, loopless graph $\hhG$; indeed, as in   \cite{AC}, the combinatorial rank of $\md$ is
$$r_G(\md):=r_{\hhG}(\hhd),$$
where $r_{\hhG}(\hhd)$ is defined,  as  in \cite{BNRR}  (for weightless and loopless graphs), as follows.
If $\hhd$ is not equivalent to any effective divisor, we set $r_{\hhG}(\hhd)= -1$;
otherwise
$r_{\hhG}(\hhd)$ is the maximum   integer $k\geq 0$ such that for every effective
$\me\in \Div(\hhG)$ of degree $k$ the divisor $\hhd-\me$ is equivalent to an effective divisor.
In particular, if  $|\hhd|<0$
then $r_{\hhG}(\hhd)= -1$.

\

Now, let $g$ and $e$ be two non-negative integers; we define
 $$
e^g=e+\min\{ e,g\}.
 $$
\begin{remark}
\label{inflated}
For every irreducible curve $C$   of genus $g$,
by the theorems of Riemann-Roch and Clifford, $e^g$ is the minimum degree of a line bundle of rank $e$; more precisely
 for  a line bundle  $L$ on $C$ we have: $$
 r(C,L)\geq e \  \Longrightarrow \  \deg L\geq e^g.
 $$
 \end{remark}
The previous remark  together with the subsequent Lemma~\ref{lemmaew} serve  as motivation for the next definition.

\begin{defi}
\label{defmeg} Let $G$ be a   graph and $\me$  an effective divisor  of $G$.
We  define the (effective) divisor $\meg$ on $G$ so that for every $v\in V$
\[
\meg(v)= \me(v)^{g(v)}   =\me(v)+\min\{ \me(v),g(v)\}.
\]
\end{defi}
The superscript ``$\deg$"   indicates that the entries of $\meg$ are ``minimum degrees" associated to the entries of $\me$, as explained in Remark~\ref{inflated}.

Notice that if $G$ is weightless and loopless then $\meg= \underline{e}$.

\begin{lemma}
\label{lemmaew}
Let $G$ be a  graph and let $\md\in \Div(G)$.
If for every effective divisor $\me$ of degree $r$ the divisor $\md - \meg$
is equivalent to an effective divisor, then $r_G(\md)\geq r$.
\end{lemma}

\begin{proof}
Let $\hhG$ be the weightless loopless graph  defined at the beginning of the subsection; we need to show that $r_{\hhG}(\hhd)\geq r$, where $\hhd$  is the  divisor induced by $\md$ on $ \hhG$; see     \eqref{Divmap}.

So, let $\epsilon\in\Div(\hhG)$ be any effective divisor of degree $r$, and let us prove that
$\hhd-\epsilon$ is  equivalent to an effective divisor.
 We define an effective divisor $\me\in\text{Div}(G)$ by setting, for every vertex $v$ of $G$,
\begin{equation}
 \label{evdef}
\me(v):= \epsilon(v) + \epsilon(z_v^1) + \ldots + \epsilon(z_v^{g(v)})
\end{equation}
where   $z_v^1,\ldots, z_v^{g(v)}\in V(\hhG)$ as  in \eqref{Rv}.
Note that $|\me|=|\epsilon|=r$  hence, by hypothesis, $\md-\meg$ is equivalent to an effective divisor. Therefore, it suffices to show that $\widehat{\meg}$ is equivalent to a divisor containing $\epsilon$. For every  vertex $v\in\text{Div}(G)$   consider the principal divisor $\mt^v\in \Prin(\hhG)$
$$\mt^v := -\sum_{k=1}^{g(v)}a_v^k \mt_{ z_v^k},$$  where $\mt_{ z_v^k}\in \Prin(G)$ are defined in \eqref{tZdef} and,  for every $1\leq k\leq g(v)$,
$$
a_v^k := \Big\lceil\frac{\epsilon(z_v^k)}{2}\Big\rceil.
$$
To prove the lemma   we will show that     $\widehat{\meg}+ \sum_{v\in V(G)}\mt^v$  is effective and contains $\epsilon$.
We have, for any $u\in V(G)$,
\begin{eqnarray*}
\Bigr(\widehat{\meg}+ \sum_{v\in V(G)}\mt^v\Bigr)(u)&=& \me(u)+\min\{\me(u), g(u)\}- 2 \sum_{k=1}^{g(u)}a_u^k \\
&=& \me(u)+\min\{\me(u), g(u)\}- \sum_{k=1}^{g(u)}\epsilon(z_u^k)   - o(u)\\
 \end{eqnarray*}
where
$ o (u)$ denotes the number of indices $k$ with $1\leq k\leq g(u)$ such that
$\epsilon (z_u^k)$ is odd.
We have
$
  o (u)\leq \min\{\me(u), g(u)\},
$
hence, using \eqref{evdef},
$$
\Bigr(\widehat{\meg}+ \sum_{v\in V(G)}\mt^v\Bigr)(u)=  \epsilon(u) +\min\{\me(u), g(u)\}- o(u)
 \geq  \epsilon(u),
$$
as required. Next,
$$
\Bigr(\widehat{\meg}+ \sum_{v\in V(G)}\mt^v\Bigr)(z_u^k)=2 \Big\lceil\frac{\epsilon(z_u^k)}{2}\Big\rceil\geq \epsilon(z_u^k).
$$
The lemma is proved. \end{proof}

\begin{remark} Our proof of the previous Lemma has the advantage of being   elementary  and self-contained.
But  an alternative proof   could be obtained using metric graphs, as follows.
By \cite[Lemma 1.5]{bakersp}, when computing the rank of a divisor, we may regard $G$ as a metric graph. By \cite[Proposition 2.5]{Yoav}, the vertices of $G$ are a weighted rank determining set, which, by definition, means that a divisor $\md$ has rank at least $r$ exactly when $\md - \meg$ is equivalent to an effective divisor for every effective divisor $\me$ of degree $r$ supported on the vertices. But this is true by assumption.
\end{remark}

\subsection{Weightless and loopless graphs} An important tool for computing combinatorial ranks is the notion of reduced divisors with respect to a vertex as introduced by Baker and Norine in \cite{BNRR} (see also \cite{Luo}).
In this subsection we consider weightless graphs with no loops, as in loc. cit.;
let us recall the definition and a few basic facts.

\begin{defi}\label{reduced}
Let $\md$ be a divisor on a weightless and loopless graph $G$ and fix a vertex  $u\in V$.
A divisor $\md$ is said to be {\it reduced} with respect to $u$, or $u$-{\it reduced},  if
\begin{enumerate}[(1)]
\item $\md(v)\geq 0$ for all $v\in V\smallsetminus \{u\}$;
\item  for every non-empty set $A \subset V\smallsetminus\{u\}$, there exists a vertex
$v \in A$ such that $ \md(v) < v\cdot A^c$.
\end{enumerate}
Conditions (1) and (2)  are equivalent to the following:
\begin{enumerate}[$(1')$]
\item $\md_{|V\smallsetminus \{u\}} \geq 0$;
\item  for every non-empty  $A \subset V \smallsetminus\{u\}$
we have
 $(\md +\mt_A)_{|V\smallsetminus \{u\}}\not\geq 0$.
 \end{enumerate}
 \end{defi}
\begin{remark} Let $\md$ be a divisor on $G$ and   $u$ be a vertex of $G$.
\label{basicred}
\begin{enumerate}[(a)]
 \item
 If $\md$ is $u$-reduced, then for any integer $n$ the divisor $\md-nu$ is also $u$-reduced.
 \item
 If $\md$ is $u$-reduced on the graph $G$, and   $G$ is a spanning subgraph of $G'$
(i.e. $V(G)=V(G')$), then $\md$ is $u$-reduced on $G'$.
\end{enumerate}
\end{remark}
\begin{fact} \cite[Proposition 3.1]{BNRR}\label{redfct} Let $G$ be a weightless and   loopless graph. Fix a vertex $u \in V (G)$.
Then for every $\md \in\Div(G)$ there exists a unique $u$-reduced divisor $\md' \in \Div(G)$ such that $\md' \sim \md$.
\end{fact}
The next result is probably well known to the experts, but we could not find a proof in the literature.
\begin{lemma}\label{E:reduced} Let $\delta\in \Pic(G)$.
\begin{enumerate}[(a)]
 \item
 \label{reda}
$r_G(\delta)=-1$    if and only if there exists a vertex $u$  whose $u$-reduced representative $\md$ in $\delta$ has $\md(u)<0$.
\item
\label{redb}  $r_G(\delta)=0$    if and only if there exists a vertex $u$  whose $u$-reduced representative $\md$ in $\delta$ has $\md(u)=0$.
\end{enumerate}
\end{lemma}
\begin{proof}
 One implication of part \eqref{reda}  is clear. To prove the other,  let $\md$ be $u$-reduced and suppose $\md(u)<0$.
By contradiction, suppose  $r_G(\md)\geq 0$; let $\mt \in \Prin(G)$ be such that $\md+\mt \geq 0$, notice that $\mt$ is not trivial.
By Remark~\ref{tZrk}, there exists a non-empty $Z\subsetneq V(G)$ such that $\mt_{|Z}\leq  (\mt_Z)_{|Z}$, hence
$$
0\leq (\md+\mt)_{|Z} \leq (\md + \mt_Z)_{|Z}.
$$
In particular, $u\not\in Z$ (if $u\in Z$ then $\mt _Z(u)\leq 0$,  hence the above inequality yields $0\leq \md(u)$,
contradicting the initial assumption).
For any $v\neq u$ with $v\not\in Z$ we have $\mt_Z(v)\geq 0$, hence
$$
(\md + \mt_Z)(v)\geq \md(v)\geq 0,
$$
because $\md$ is $u$-reduced. We have thus proved that $(\md + \mt_Z)_{|V\smallsetminus \{u\}}\geq 0$,
which is impossible as   $\md$  is $u$-reduced.

For \eqref{redb}, suppose $\md$ is $u$-reduced with $\md(u)=0$; then $r_G(\md)\geq 0$. As the divisor $u\in \Div (G)$ has degree 1,  it suffices to prove that $r_G(\md -u)=-1$. Remark~\ref{basicred} yields that $\md-u$ is $u$-reduced, moreover  $(\md-u)(u)=-1$.
By the previous part we get   $r_G(\md-u)=-1$.

Conversely, assume $r_G(\delta)=0$.
By the previous part,  if $\md\in \delta$ is $u$-reduced, then $\md(u)\geq 0$.
It suffices to show that if,
for all $u\in V$, the $u$-reduced representative, $\md$,  satisfies $\md(u)>0$, then $r_G(\delta)\geq 1$.
 Indeed, let $\me$ be an effective divisor of degree $1$; hence $\me=v$ for some $v\in V$.
 Let   $\md'\in\delta$ be $v$-reduced, then $\md'-v$ is effective, and we are done.
\end{proof}

  Let $G$ be a loopless, weightless graph and pick $\md\in \Div G$.
Set
\begin{equation}
\label{ad}
\agd=\left\{ \begin{array}{ll}
\min\{ \md(v),\  \  \forall v\in V(G)\}  \  &\text{ if }   \md \geq 0\\
\  -1 \  &\text{ otherwise.}\\
\end{array}\right.
\end{equation}
\begin{remark}
\label{adrk}
It is clear that $r_G(\md)\geq \agd$ (in fact,   ``$\ell$" stands for ``lower  bound"). We now look for conditions under which equality occurs.
\end{remark}

The following result generalizes to arbitrary combinatorial rank the implications ``if" of Lemma~\ref{E:reduced}
(the implication ``only if" does not generalize; see Example~\ref{notrex}).
\begin{prop}
\label{rank}  Let $G$ be a weightless, loopless graph.
Let $\md \in \Div (G)$ be such that for some $u\in V(G)$ with $\md(u)= \agd$ we have that $\md$ is reduced with respect to $u$.
Then $r_G(\md)=\agd$.
\end{prop}
\begin{proof} For simplicty, we write  $\ell=\agd$.
Let us first suppose $\md\not\geq 0$, that is $\ell=-1$. Then the statement is a consequence of Lemma~\ref{E:reduced}.

Assume now $\ell\geq 0$.
Since $r_G(\md)\geq \ell$  it suffices to prove
\begin{equation}
\label{thr}
r_G(\md)< \ell+1.
\end{equation}
 Let $\me:=(\ell+1)u\in \Div (G)$, with $u$ as in the statement; so $\me$ is an effective
divisor of degree $\ell+1$.  Set $\mc=\md -\me$;
to prove \eqref{thr} it is enough to show that
$r_G(\mc)=-1$.
Now,   $\mc$ is reduced with respect to $u$ (see Remark~\ref{basicred}), hence   the previous case yields $r_G(\mc)=-1$.
\end{proof}

\subsection{Computing the rank for general  graphs}
We now generalize the previous set-up   to weighted   graphs.
Let $g$ and $d$ be two non-negative integers; set
 $$
d_g:=\max\left \{d-g,\left\lfloor{\frac{d}{2}}\right\rfloor\right \}.
 $$
\begin{remark}
\label{rdg}
For every irreducible curve $C$   of genus $g$,
by the theorems of Riemann-Roch and Clifford, $d_g$ is the maximum rank of a line bundle of degree $d$; more precisely
 for  a line bundle  $L$ on $C$ we have: $$
\deg L \leq d \  \Longrightarrow \   r(C,L)\leq d_g.
 $$
 Similarly,
  $d_g$ is the maximum combinatorial rank of a divisor of degree $d$ on a  graph $G$   of genus $g$; that is,  for   every $\md \in \Div(G)$ we have
 $$
|\md| \leq d \  \Longrightarrow \   r_G(\md )\leq d_g.
$$
This follows from the Riemann-Roch formula, \cite[Thm 3.8]{AC}, and the Clifford inequality,
due to Baker Norine \cite[Corollary 3.5]{BNRR} (the extension  to weighted graphs is trivial, see \cite[Proposition 1.7(4)]{Cjoe}).
 \end{remark}
 The notation $d_g$ introduced here is related  to the notation $e^g$ introduced before Remark~\ref{inflated}
 by the following trivial Lemma, of which we omit the proof.
\begin{lemma}
 \label{triv}
 Let $g,e,d$ be non-negative integers.
 Then $(e^g)_g=e$ and
\begin{displaymath}
(d_g)^g=\left\{ \begin{array}{ll}
d-1  \  &\text{ if }   d\leq 2g-1 \text{ and $d$ is odd}  \\
d \  &\text{ otherwise.}\\
\end{array}\right.
\end{displaymath}
\end{lemma}

Now we define:
\begin{defi}
\label{defmdr}Let $G$ be a   graph  and let $\md\in \Div G$ be an effective  divisor.
We  define the (effective) divisor $\mdr$ such that for every $v\in V$
$$
\mdr(v)= \md(v)_{g(v)}=\max\left \{ \md(v)-g(v),\left\lfloor{\frac{\md(v)}{2}}\right\rfloor\right \}.
$$
\end{defi}

\begin{remark} The subscript ``$\operatorname{rk}$" indicates that the entries
of $\mdr$ are thought of as ``maximum ranks", in the spirit of Remark~\ref{rdg}.
\end{remark}
\begin{example}
 \label{svrk}  Suppose $G$ is a graph with no edges, made of a single vertex $v$ having weight $g(v)$. Then for any
   $\md \in \Div(G)$ we have $r_G(\md) =\mdr(v)$.
More exactly, we have $\hhG = R_v$ (where $R_v$
is the  graph   introduced in \eqref{Rv})
and
$$
r_G(\md) =r_{R_v}(\hhd) = \md(v)_{g(v)}
$$
see    \cite[Lemma 3.7 and Thm 3.8]{AC}.
\end{example}
For any divisor $\md$ on $G$ we define
\begin{equation}
\label{adw}
\agd=\left\{ \begin{array}{ll}
\min\{\mdr(v),\  \  \forall v\in V(G)\}  \  &\text{ if }   \md \geq 0\\
\  -1 \  &\text{ otherwise.}\\
\end{array}\right.
\end{equation}

This definition of $\agd$ coincides with the one in \eqref{ad} if  $G$ has no loops and no weights.

\begin{lemma}
\label{adwprop}
Let $G$ be any graph and let $\md\in \Div G$ be a divisor. Then $r_G(\md)\geq \agd$.
\end{lemma}
\begin{proof}
Set $\ell= \agd$; we can clearly assume $\ell\geq 0$, hence $\md \geq 0$.
Note that we have
$$\md(v)\geq \mdr (v)+\min\{\mdr(v),g(v)\}=  \mdr(v)^{g(v)}
$$
 for every $v\in V(G)$.
Now, let $\me$ be any effective divisor of degree $\ell$.
  In particular, $\me(v)\leq \ell $  for every $v\in V(G)$, so by the definition of $\ell$, we have
$$\mdr(v)\geq \ell\geq \me(v)$$ hence
$$
\mdr(v)^{g(v)}\geq \me(v)^{g(v)}.
$$
Combining the last inequality with the first we have
$$\md(v)\geq \mdr(v)^{g(v)}\geq \meg(v).$$
We conclude  that for every effective divisor $\me$ with $|\me|=\ell$ we have
$\md-\meg \geq 0$.
By Lemma~\ref{lemmaew} we are done.
\end{proof}

Let   $G$ be any graph.
As in \cite{AC}, we denote by $G_0$ the weightless graph  obtained by removing from $G$ every loop, and disregarding all weights.
We have   natural identifications $V(G)=V({G_0})$
and  $\Div(G)=\Div(G_0).$
Now,    this identification does not preserve the combinatorial rank, but we have
 \begin{equation}
\label{rG}
r_{G_0}(\md)\geq   r_{G}(\md)
\end{equation}
 for every $\md\in \Div(G)$
(see \cite[Remark. 3.3]{AC}).
Also, we have  $\Prin ({G_0})=\Prin (G)$
and hence an identification
$
\Pic (G)=\Pic(G_0).
$
Finally, the above facts are easily seen to   imply
\begin{equation}
\label{r-1}
r_{G_0}(\md)=-1\Leftrightarrow r_{G}(\md)=-1.
\end{equation}
 The next result generalizes Proposition \ref{rank}.

\begin{prop}
\label{wrank}  Let $G$ be any graph.
Let $\md \in \Div (G)$ be such that for some $u\in V(G)$ with $\mdr(u)= \agd$ we have that $\md$ is reduced with respect to $u$.
Then
$r_G(\md)=\agd$.
\end{prop}

\begin{proof}
Let us write $\ell=\agd$ and let $G_0$ be the weightless, loopless graph defined above. We first suppose   $\md\ngeq 0$, i.e., $\ell=-1$. Then the fact  that $r_G(\md)=-1$ follows immediately by combining   \eqref{r-1} with Lemma~\ref{E:reduced}.

Assume now that $\ell\geq 0$. By Lemma~\ref{adwprop}  it is enough to show that
$$r_G(\md)< \ell+1.$$

Denote by
 ${G^u}$
  the graph obtained from $G$ by adding $\w(u)$ loops based at $u$ and then by inserting a vertex
 in each loop based at  $u$. 
  Notice that, using the notation at the beginning of subsection~\ref{wsec},
we have a connected subgraph $R_u\subset G^u$
of genus $g(u)$.
Denote by $\md^u$   the divisor obtained by extending $\md$ to $G^u$ with degree $0$ on the new vertices of $R_u$.

As we saw in Example~\ref{svrk},  we have
\begin{equation}
\label{Ru}
\ell=\mdr(u)=r_{R_u}( \md^u_{|R_u}).
\end{equation}
It is clear that
 $r_G(\md)=r_{G^u}(\md^u)$
(as $\widehat{{G^u}}=\hhG$ and $\widehat{\md^u}=\hhd$). It suffices to show  $r_{G^u}(\md^u)<\ell+1$.

By \eqref{Ru}, there is on $R_u$  an effective divisor, $\un e_u$,    such that   $|\me_u|=\ell+1$ and $$r_{R_u}({\md^u}_{|R_u}-\un e_u)=-1.$$
Define $\me\in \Div(G^u)$  as the extension of $\me_u$ to $G^u$ with degree $0$ on all vertices outside $R_u$ and set $\mc:=\md^u-\me$.
 Since  $|\me|=\ell+1$, to conclude  it is enough to check that $r_{G^u}(\mc)=-1$.

Let $\mc'_u\in \Div(R_u)$ be linearly equivalent to  ${\md^u}_{|R_u}-\un e_u$   and reduced with respect to $u$;  since $r_{R_u}(\mc'_u)=-1$ we have  $ \mc_u'(u)<0$.
Let $\mc'$ be the extension of $\mc'_u$ to $\Div(G^u)$ such that $\mc'(v):=\mc(v) =\md(v) $ for all $v\in V(G^u)\smallsetminus V(R_u)$; then $\mc'$ is linearly equivalent to $\mc$ (since we have an inclusion $\Prin (R_u)\hookrightarrow \Prin(G^u)$),
hence it suffices to prove that
$$
r_{G^u}(\mc')=-1.
$$
Now, the restrictions of $\mc'$ to $R_u$ and to $G$ are both $u$-reduced
(the latter  because $\mc'$ coincides with $\md$ on $V(G)\smallsetminus \{u\}$, which is  $u$-reduced   by hypothesis).
Hence
$\mc'$ is $u$-reduced, and hence $r_{G^u}(\mc')=-1.$
\end{proof}

\subsection{Dhar decomposition}
 \label{S:Dhar}
Fix any graph $G$; we are going to define a decomposition, which we call the ``Dhar decomposition", that    will come very useful with inductive arguments, and which is  independent of the weights or of the loops of $G$.

Fix a vertex $u$ of $G$ and  let $\md$ be a divisor on $G$ whose restriction to $V\smallsetminus\{u\}$ is effective;  the   Dhar  decomposition associated to $\md$    with respect to    $u$ is a decomposition of $V$   that will be   denoted as follows:
\begin{equation}
 \label{dhareq}
V=Y_{0}\sqcup Y_1\ldots\sqcup Y_{l}\sqcup W.
\end{equation}
 For the reader familiar with Dhar Burning Algorithm, $Y_j$ will be the set of vertices burned at the $j$-th  day when starting a fire from $u$, and $W$ is the set which remains unburned at the end. Hence, we shall refer to
it as the Dhar decomposition of $V$ associated
to $u$.

Denote $Y_{0}=\{u\}$ and set $W_0= V\smallsetminus \{u\}$. If $\md+\un t_{W_0}$
is effective (which implies that $\md$ is not reduced at $u$), then set $W=W_0$, and the decomposition is just
 $V= Y_0\sqcup W.$
Otherwise, define $Y_{1}$ to be the set of vertices in $W_0$ where $\underline{d}+\underline{t}_{W_0}$
is negative.

Now, repeat the process: suppose that $Y_{0},Y_{1},\ldots,Y_{j-1}$
have been defined. Denote $W_{j-1}:=V\smallsetminus Y_{0}\sqcup\ldots\sqcup Y_{j-1}$, and consider $\md+\mt_{W_{j-1}}$.
If it is effective, then $W$ will be equal to  $W_{j-1}$,
and we are done. Otherwise, define $Y_{j}$ to be the set of vertices in $W_{j-1}$
where $\md+\mt_{W_{j-1}}$
is negative. If the process eventually exhausts all the vertices of the graph, then $W$ will be the empty set
(which occurs exactly when $\underline{d}$ is $u$-reduced; 
see \cite[Lemma 2.6]{Luo}).

\begin{remark}\label{R:decomposition}
For every $\md$ and $u$, vertices in $Y_j$ or $W$  can be characterized as follows.
 For $j=1,\dots,l -1$, we have
\begin{equation}
v\in Y_j  \Leftrightarrow v\notin Y_{j-1}   \mbox{ and } \md(v)<v\cdot (Y _{0}\sqcup Y_1\ldots\sqcup Y _{j-1}).
\end{equation}
\end{remark}
 
\begin{example}
\label{Dharex}
In the picture below we illustrate an example of a graph $G$ with vertex set $\{v_0,v_1,v_2,v_3,v_4\}$ and a divisor $\md=(0,1,2,4,4)\in \Div(G)$. Then Dhar decomposition of $\md$ with respect to $u=v_0$ is as follows
$$Y_0\sqcup Y_1\sqcup Y_2\sqcup W=\{v_0\}\sqcup \{v_1\}\sqcup  \{v_2\}\sqcup   \{v_3 ,v_4\} .$$

 On the other hand, the decomposition with respect to $u=v_4$ is simply:
 $$Y_0\sqcup W=\{v_4\}\sqcup \{v_0,v_1,v_2,v_3\}.$$

\

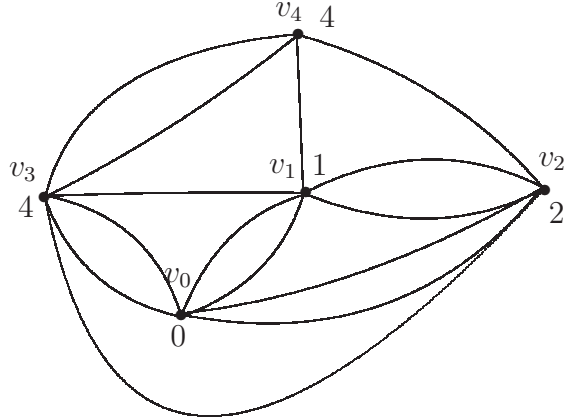
\unitlength .8mm  
\linethickness{0.4pt}
\ifx\plotpoint\undefined\newsavebox{\plotpoint}\fi  
\begin{picture}(97.25,65.75)(-10,52)
\put(13,86){\circle*{1.8}}
\put(56,87){\circle*{1.8}}
\put(95.25,87.25){\circle*{1.8}}
\put(54.75,113){\circle*{1.8}}
\put(35.5,66.5){\circle*{1.8}}
\qbezier(13.25,86.25)(30.625,82.75)(35.5,66.25)
\qbezier(13,86)(18.875,69.875)(35.25,66.25)
\qbezier(13.25,86.5)(19.25,109.75)(54.25,113)
\qbezier(13.5,86.25)(35,97)(54.5,112.75)
\qbezier(54.5,113)(79.625,105.875)(95.25,87.25)
\qbezier(54.5,113)(55,101.375)(55.5,87.25)
\qbezier(13,86.25)(35.125,87)(55.75,86.75)
\qbezier(55.75,86.75)(76.5,97.375)(95.25,87.5)
\qbezier(95.25,87.5)(77.375,78)(56,86.5)
\qbezier(56,86.5)(41.75,82.625)(35.5,66.25)
\qbezier(35.5,66.25)(51.375,71.875)(55.75,87)
\qbezier(35.5,66.75)(64.375,69.875)(94.75,87.5)
\qbezier(94.75,87.5)(72.25,59.5)(35.75,66.5)
\qbezier(13,86.5)(26.625,12.5)(95.25,87.5)
\put(35,72.75){\makebox(0,0)[cc]{$v_0$}}
\put(52.25,90.75){\makebox(0,0)[cc]{$v_1$}}
\put(96.5,92){\makebox(0,0)[cc]{$v_2$}}
\put(53.25,116.75){\makebox(0,0)[cc]{$v_4$}}
\put(9.5,90.25){\makebox(0,0)[cc]{$v_3$}}
\put(10,84.5){\makebox(0,0)[cc]{$4$}}
\put(35,63){\makebox(0,0)[cc]{$0$}}
\put(58.25,91.5){\makebox(0,0)[cc]{$1$}}
\put(97.25,83.5){\makebox(0,0)[cc]{$2$}}
\put(59.5,116){\makebox(0,0)[cc]{$4$}}
\end{picture}
\end{example}

Let $\md\in \Div(G)$ be an effective  divisor   which is non-reduced with respect to a fixed vertex $u$ of $G$.
Then, by adding a suitable set of edges to $G$, one can construct a graph on which $\md$ is $u$ reduced.
\begin{defi}
\label{Dhardef}
Let  $u$ be a vertex of a graph   $G$. Let $\md\in \Div(G)$ be such that $\md(v)\geq 0$ for every $v\in V(G)\smallsetminus \{u\}$.
A {\it saturation} of $G$ (or a $G$-{\it saturation}) with respect to $u$ and $\md$ is a graph $G'$ satisfying the following requirements:
\begin{enumerate}[(a)]
 \item
 $G$ is a spanning subgraph of $G'$;
\item
$\md$ is $u$-reduced as a divisor on $G'$;
\item
every edge in $E(G')\smallsetminus E(G)$ is adjacent to $u$.
\end{enumerate}
\end{defi}
For instance, in   Example~\ref{Dharex},
 a $G$-saturation     with respect to $\md$  and $v_0$
can be obtained by adding one edge  between $v_0$ and $v_3$ and  one edge  between $v_0$ and $v_4$.
And a $G$-saturation     with respect to $\md$  and $v_4$    is given adding one edge between $v_0$ and $v_4$, or also one edge between
between $v_1$ and $v_4$.

\begin{remark} \label{Dharrk}
It is trivial to check that such saturations always exist.

More precisely, consider the Dhar decomposition \eqref{dhareq} with respect to $\md$ and $u$.
Then, as $G$ is connected,  by adding $\md(w)$ edges between $w$ and $u$  for every $w\in W$, we get a $G$-saturation with respect to $\md$ and $u$.
\end{remark}

 \begin{prop}
\label{GG'}
Let $\md$ be an effective divisor on a    loopless, weightless graph $G$, and
let $u\in V(G)$ be such that $\md(u)=\agd$.
Let $G'$ be a  $G$-saturation with respect to $u$ and $\md$. Set $m=|E(G')\smallsetminus E(G)|$; then
 $$
r_G(\md)\leq   \agd+m
 $$
 (equivalently:   $r_G(\md)\leq   r_{G'}(\md)+m$).
  \end{prop}

\begin{proof}
Write $\ell =\agd$.
Recall that the edges in $E(G')\smallsetminus E(G)$ are all adjacent to $u$;
denote by $\{w_1,\ldots,w_k\}\subset V(G)\smallsetminus \{u\}$
the vertices  adjacent to the different edges in $E(G')\smallsetminus E(G)$. Set
\begin{equation}
 \label{mm'}
 m_i=(u\cdot w_i)_{G'}-(u\cdot w_i)_{G}
\end{equation}
(the subscripts $G$ and $G'$ indicate  the graph where the intersection product is computed).
We have  $\sum_{i=1}^k m_i=m$.

Write $$\mc=\md -(\ell+1)u-\sum_{i=1}^km_iw_i;$$ we have $\mc(u)=-1$.
It suffices to prove  that the rank of $\mc$ is $-1$.  By contradiction, suppose this is not the case;
then there exists $\mt\in \Prin(G)$ such that
$\mc+\mt\geq 0$. According to
 Remark~\ref{tZrk} we have
 $\mt=\sum_{i=1}^k i\cdot \mt_{Z_i}$ where
  $Z_1,\ldots Z_k$ are disjoint   sets of vertices with $Z_k\neq \emptyset$; by the same remark we have $\mt _{|Z_k} \leq (\mt_{Z_k})_{|Z_k}$. Summarizing:
\begin{equation}
 \label{summ}
0\leq  (\mc + \mt)_{|Z_k} =\mc_{|Z_k} + \mt _{|Z_k} \leq \mc_{|Z_k} + (\mt_{Z_k})_{|Z_k} \leq \mc_{|Z_k}
  \end{equation}
 (as $ (\mt_{Z_k})_{|Z_k}\leq 0$).  Hence  $u\notin Z_k$.  Write $\mt'_{Z_k}$ for the principal divisor corresponding to $Z_k$ in $G'$.
Pick $v\in Z_k$; if $v=w_i$ for some $i$, then $\mt'_{Z_k}(v) = \mt_{Z_k}(v) - m_i$ by \eqref{mm'}; otherwise we have  $\mt'_{Z_k}(v) = \mt_{Z_k}(v)$.
In either case for every $v\in Z_k$ we have
$$
(\md  + \mt'_{Z_k})(v) = (\mc + \mt_{Z_k})(v) \geq 0
$$
by \eqref{summ}; a contradiction to the reducedness of $\md$ on $G'$.
\end{proof}

\noindent
{\bf Example~\ref{Dharex} continued.}
Let us show that in Example \ref{Dharex} we have $r_G(\md)=2$.
Of course, $\ell_G(\md)=\md(v_0)=0$ and, as we already observed, a $G$-saturation with respect to $v_0$ can be obtained by adding two edges on $G$. Hence Proposition \ref{GG'}  applies with $m=2$, giving $r_G(\md)\leq 2$.
Let us now  check that $r_G(\md)\geq 2$. We have
$$
\md+\mt_{\{v_3,v_4\}} =(0,1,2,4,4)+(2,2,2,-4,-2)=(2,3,4,0,2)=:\md'.
$$
Hence it remains to check that $\md-(v_0+v_3)$ is equivalent to an effective divisor, which we do as follows:
$$
\md'+\mt_{\{v_1,v_2,v_4\}}=(2,3,4,0,2)+(4,-3,-3,4,-2)=(6,0,1,4,0).
$$

\section{The inequality $r_G(\delta)\geq \rgdel$}
\label{insec}
\subsection{Proof of the inequality}

Suppose $L$ is a line bundle on a  curve $X$ such that its multidegree
corresponds to a $u$-reduced divisor for some vertex $u$ of the dual graph of $X$. Then a  non-zero section of $L$   cannot vanish identically  on the component of $X$ corresponding to $u$. This is a special case of the next
lemma (namely the case when   $\md$ is $u$-reduced), which is similar to Lemma 4.9 in \cite{AC}.
Recall that for an effective divisor $\me$, we   introduced the effective divisor $\meg$  in Definition~\ref{defmeg}.
\begin{lemma}
\label{P:sections}Let $X$ be a nodal curve whose
dual graph is $G$. Let $L$ be a line bundle on $X$, and denote $\underline{d} = \underline{\deg}L$.
Suppose that for some   $u\in V(G)$, and  effective  divisor  $\underline{e}\in \Div(G)$,
the divisor $\underline{d}-\meg$ is $u$-reduced.
Then the space of global sections of $L$ vanishing identically on
$C_{u}$ has dimension at most $|\underline{e}|-\me(u)$.\end{lemma}

\begin{proof}
Consider   the Dhar decomposition associated to
$\underline{d}-\meg$ with respect to  $u$   as in \ref{S:Dhar}; since
 $\underline{d}-\meg$ is $u$-reduced, $W$ is
empty and $V=Y_{0}\sqcup\ldots\sqcup Y_{l}$.

For each $0\leq j\leq l$, denote by $\Lambda_{j}\subset H^0(X,L)$ the space of
sections of $L$ vanishing on the components of $X$ corresponding
to the vertices of $Y_{0}\sqcup\ldots\sqcup Y_{j}$.  We must prove that $\dim \Lambda_0 \leq  |\underline{e}|-\me(u)$.
We will proceed  by
descending induction on $0\leq j\leq l$, showing that
\begin{equation}
 \label{Lambdaj}
 \dim\Lambda_{j}\leq \sum_{i=j+1}^{l}|\underline{e}_{|Y_{i}}|.
\end{equation}
For $j=l$, the claim is obvious, since $\Lambda_{l}$ is the space of
sections vanishing on the entire curve, and its dimension is $0$.
Now, assume   that \eqref{Lambdaj} holds and let us
consider $\Lambda_{j-1}$.

Let $v$ be any vertex of $Y_{j}$, and let $D_v$ be
the divisor on $C_{v}$ consisting exactly of the intersection points
of $C_{v}$ with the components of $X$ corresponding to $Y_{0}\sqcup\ldots\sqcup Y_{j-1}$. By Remark~\ref{R:decomposition}, we have
$$
(\underline{d}- \meg)(v) - \deg(D_v) < 0.
$$
Hence
$$
\deg_{C_v}L(-D_v)<\meg(v),
$$
and therefore, by Remark~\ref{inflated},   we get
$$h^{0}(C_{v},L(-D_{v}))\leq\underline{e}(v).$$
We obtain
\begin{equation}
 \label{Bigoplus}
\dim \Bigl(\underset{v\in Y_{j}}{\bigoplus}H^{0}(C_{v},L(-D_{v}))\Bigr)\leq \sum_{v\in Y_{j}}\underline{e}(v)=|\underline{e}_{|{Y_{j}}}|.
\end{equation}
Now, consider the exact sequence
\[
0\to\Lambda_{j}\rightarrow\Lambda_{j-1}\overset{\alpha}{\rightarrow}\underset{v\in Y_{j}}{\bigoplus}H^{0}(C_{v},L(-D_{v}))\mbox{,}
\]
where $\alpha$ is the map restricting a section to each component. Then
$$\dim\Lambda_{j-1} \leq  \dim\Lambda_{j}+\dim \Bigl(\underset{v\in Y_{j}}{\bigoplus}H^{0}(C_{v},L(-D_{v}))\Bigr)\leq\sum_{i=j}^{l}|\underline{e} _{|Y_{i}}|,$$
 where the last inequality follows from the induction hypothesis and \eqref{Bigoplus}. The proof is complete.
\end{proof}

\begin{thm}\label{rankthm}
Let $\delta$ be a divisor class on a graph $G$. Then
\[
\ralg(G,\delta)\leq r_{G}(\delta)\mbox{.}
\]
\end{thm}
\begin{proof}
Denote $s=\ralg(G,\delta)$. If $s=-1$ then the claim is obvious, since the combinatorial rank is always bounded below by $-1$. Hence we may assume that $s\geq 0$.

In order to prove that $r_G(\delta) \geq s$, by Lemma~\ref{lemmaew} it suffices  to show that for any effective divisor $\underline{e}$
with $|\me|=s$, the divisor class $\delta$ admits a representative $\md$ such that  $\md -\meg\geq 0$.

Let  $\underline{e}$ be such a divisor, and let $u\in V$ be a fixed vertex.
  Using Fact~\ref{redfct} we have that there exists a representative $\md$ for $ \delta$
 such that $\underline{d}-\meg$
is $u$-reduced. Now,  by definition, $\underline{d}-\meg$
is effective on $V\smallsetminus \{u\}$, so it remains to show that
\begin{equation}
 \label{34thesis}
\underline{d}(u)-\meg(u)\geq 0.
\end{equation}
Since $\ralg(G,\delta)=s$,
there exists a curve $X\in \MaG$ and
 a line bundle $L\in \Pic^{\md}(X)$ such that  $h^0(X,L)\geq s+1$.
Consider the exact sequence
\[
0\to\mbox{ker}(\pi)\to H^{0}(X,L)\overset{\pi}{\to}H^{0}(C_{u},L_{C_u})\mbox{,}
\]
where $\pi$ is the restriction of sections to $ C_{u} $.
The space $\ker(\pi)$ is exactly the set of global sections of $L$ vanishing on $C_{u}$, so by Lemma~\ref{P:sections},
$$\dim \ker(\pi)\leq s- \me(u).$$ From the above exact sequence we obtain
$$
h^{0}(C_u,L_{C_{u}})\geq
h^{0}(X,L)-\dim(\ker\pi)\geq s+1-s+\me(u)=\me(u)+1.
$$
Now Remark~\ref{inflated} yields
$$\deg _{C_{u}} L\geq \me(u)^{g(C_u)},$$
in other words $\underline{d}(u)\geq\meg(u)$, which proves \eqref{34thesis} and the theorem.
\end{proof}
\begin{cor}
\label{cor0-1} Let $\delta$ be a divisor class on  $G$ such that  $r_G(\delta)\leq 0$. Then  
 for every $X\in \MaG$ we have $r(X,\delta)=r_G(\delta).$
In particular, 
$\ralg(G,\delta)=r_G(\delta).$
\end{cor}
\begin{proof}
The case  $r_G(\delta)= -1$  is obvious. Suppose  $r_G(\delta)= 0$; by the previous theorem it suffices to show that for any
$X\in \MaG$ we have $r(X,\delta)\geq 0$ .
Let  $\md\in\delta$ be such that $\md\geq 0$, now let $L=\mathcal O_X(D)$ where $D$ is an effective Cartier divisor of multidegree $\md$; then
 $r(X,L)\geq 0$ and hence
  $\rmax(X,\md)\geq 0$.

Now, let $\mc\in\delta$ be a different representative for $\delta$ and write $\mc+\mt=\md$,
for some  non-trivial $\mt\in \Prin(G)$. By Remark~\ref{tZrk},  we have
\begin{equation}
 \label{mtZ0}
\mt_{|Z}\leq (\mt_{Z})_{|Z},
\end{equation}
for some non-empty $Z\subsetneq V$.  We shall abuse notation and denote by $Z$ and $Z^c$ the subcurves of $X$  whose components correspond to the vertices in $Z$
and $Z^c$ respectively.
We have $r(Z, L_Z)\geq 0$, of course.

Let
  $Z\cdot Z^c\in \Div(Z)$ be the  (effective, Cartier) divisor cut on $Z$ by $Z^c$;
 its multidegree, $ \mdeg_ZZ\cdot Z^c$, satisfies
\begin{equation}
 \label{mZZ0}
  \mdeg_ZZ\cdot Z^c= (\mt_{Z^c})_{|Z}=-(\mt_{Z})_{|Z}.
\end{equation}
  Now,   for every  $M\in \Pic ^{\mc}(X)$  we have
\begin{equation}
 \label{MMc}
r(X,M)\geq r(Z, M_Z(-Z\cdot Z^c))
\end{equation}
 (any section of $M_Z(-Z\cdot Z^c)$ vanishes on $Z\cap Z^c$ and hence can be glued to the zero section on $Z^c$).
Moreover, using \eqref{mZZ0} and \eqref{mtZ0}
$$
\mdeg _ZM(-Z\cdot Z^c)=\mc_{|Z}-\mdeg_ZZ^c= \mc_{|Z} + (\mt_{Z})_{|Z}\geq \mc_{|Z}+\mt_{|Z}=\md_{|Z}.
$$
The above inequality implies that we can  pick $M\in \Pic ^{\mc}(X)$   such that its restriction to $Z$ satisfies
$M_Z=L_Z(Z\cdot Z^c+E)$ where $E$ is some effective Cartier divisor on $Z$. By \eqref{MMc} we have
$$
r(X,M)\geq r(Z,L_Z(Z\cdot Z^c+E-Z\cdot Z^c)) \geq r(Z,L_Z)\geq 0.$$
Hence $\rmax(X,\mc)\geq 0$, and hence    $r(X,\delta)\geq 0$, as required.
\end{proof}

 \begin{cor}
Let $G$ be a  graph of genus $g$ and $\delta\in \Pic^d(G)$ with $d\geq 0$.
 If $r_G(\delta)=\min\{0,d-g\}$ then
$\rgdel= \rGdel$.
\end{cor}
\begin{proof}
By Riemann-Roch, for all $X\in \MaG$ and all $L\in \Pic^d(X)$ we have $r(X,L)\geq \min\{0,d-g\}$. By Theorem~\ref{rankthm} we are done.
\end{proof}
In \cite{KY1} and \cite[Thm 1.1 and Thm 1.2]{KY2} the authors prove the inequality $\rgdel\geq  \rGdel$   in certain cases;  combining with Theorem~\ref{rankthm}, we have the following partial answer to Problem~\ref{toy2}.
\begin{cor}
 \label{KY}
Let $G$ be a graph.
We  have $\rgdel= \rGdel$ for every   $\delta\in \Pic^d(G)$ in the following cases.
\begin{enumerate}[(a)]
\item ${\rm{char}} (k) \neq 2$ and $G$ is hyperelliptic.
\item  $G$ has genus 3 and it is not  hyperelliptic.
\end{enumerate}
\end{cor}

\subsection{Clifford inequality for the algebraic rank}
It  is well known that Clifford inequality fails trivially  for reducible curves; in fact for any
reducible curve $X$ of genus $g$ and any integer $d$ with $0\leq d\leq 2g-2$,
there exist infinitely many $L\in \Pic^d (X)$ such that $r(X,L)>\lfloor d/2\rfloor$
(see \cite[Proposition 1.7 (4)]{Cjoe}).
But consider the following question.

{\it Pick a graph $G$ of genus $g$ and $\delta\in \Pic^d(G)$ with
$0\leq d \leq 2g-2$.

\noindent
Does there exist a multidegree  $\md\in \delta$ such that for every  $X\in \MaG$ and every $L\in \Pic^{\md}(X)$ we have
$r(X,L)\leq \lfloor d/2\rfloor$?}

Apart from some special cases (see \cite{Ccliff}), the answer to this question
was not known;
we can now answer it affirmatively in full generality.

\begin{prop}[Clifford inequality]
\label{cliffcor}Let $G$ be a graph of genus $g$ and $\delta\in \Pic^d(G)$ with $0\leq d\leq 2g-2$.
Then
$$
\ralg(G,\delta)\leq \lfloor d/2\rfloor
$$
(that is $r(X,\delta)\leq \lfloor d/2\rfloor$ for every $X\in \MaG$).
\end{prop}
\begin{proof}
Immediate consequence of  Theorem~\ref{rankthm} and Clifford inequality for graphs \cite[Corollary 3.5]{BNRR}
(which, as we already mentioned, extends trivially to graphs with loops and weights).
\end{proof}

In other words, for every $X$ there exists a multidegree $\md\in \delta$ such that every $L\in \Pic^{\md}(X)$  satisfies
Clifford inequality. The following problem naturally arises.
\begin{question}
For which multidegrees does Clifford inequality hold? Do these multidegrees depend on the curve $X$, or can they be combinatorially characterized (i.e. depend only on $G$ and $d$)?
 \end{question}

A few special cases of this question are answered in \cite{Ccliff}, namely $|\md|\leq 4$, or $|V(G)|=2$. A general answer is not known.

\subsection{Reduction to loopless graphs}
Let $G$ be a graph with a loop $e$ based at the vertex $v$.
Denote by $G^{\bullet}$ the graph obtained from $G$ by inserting a weight-zero vertex, $u$,  in $e$.
There is a natural map
$$
\Div(G)\la \Div (G^{\bullet});\quad \quad \md \mapsto \md^{\bullet}
$$
such that $\md^{\bullet}(u)=0$ and $\md^{\bullet}$ is equal to $\md$ on the remaining vertices of $\md$.
The above map is a group homomorphism and  sends $\Prin (G)$ into $\Prin (G^{\bullet})$, hence
we also have  a map
$$
\Pic(G)\la \Pic (G^{\bullet});\quad \quad \delta \mapsto \delta^{\bullet}.
$$
By the definition of combinatorial rank, every $\md \in \Div(G)$ satisfies
$$
r_G(\md)=r_{G^{\bullet}}(\md^{\bullet}).
$$
Let now $X\in \MaG$ and let $X^{\bullet}\in M^{\operatorname {alg}}(G^{\bullet})$ be the curve obtained by ``blowing-up" $X$ at  the node $N_e$ corresponding to the loop $e$, i.e.
$$
X^{\bullet}=Y\cup E
$$
 where $Y$ is the desingularization of $X$ at $N_e$ and $E\cong \PP^1$ is attached to $Y$ at the branches over $N_e$.
 This process is invertible, i.e. given $X^{\bullet}$ one reconstructs $X$ by contracting the component $E$ to a node.
In conclusion, we have a bijection
$$
\MaG  \leftrightarrow M^{\operatorname {alg}}(G^{\bullet});  \quad \quad X \leftrightarrow X^{\bullet}
$$

\begin{prop}
 \label{proploopless} With the above notation, for any graph $G$, any
divisor  $\md\in \Div (G)$, any class
 $\delta \in \Pic(G)$, and any curve $X\in \MaG$  we have
 \begin{enumerate}[(a)]
\item
\label{mddelta=}
$\rmax(X,\md)=  \rmax(X^{\bullet},\md^{\bullet});$
\item
\label{deltageq}
 $r(X,\delta)\geq r(X^{\bullet}, \delta^{\bullet});$
\item
\label{alggeq}
$ \ralg (G,\delta)\geq \ralg (G^{\bullet},\delta^{\bullet}).
 $
 \end{enumerate}
\end{prop}
\begin{proof}
We begin with  \eqref{mddelta=}.
Let $\sigma:X^{\bullet}\to X$ be the morphism contracting $E$, so that its restriction to $Y$ is birational onto $X$.
Then $\sigma$ induces an isomorphism
$$
\Pic^{\md}(X)\stackrel{\cong}{\la} \Pic^{\md^{\bullet}}(X^{\bullet});\quad \quad L\mapsto \sigma^*L.
$$
For any $L\in \Pic^{\md}(X)$ we have an injection $H^0(X,L)\ha H^0(X^{\bullet}, \sigma^*L)$, therefore
$\rmax(X,\md)\leq  \rmax (X^{\bullet},\md^{\bullet})$. For the opposite inequality, pick $L^{\bullet}=\sigma^*L\in  \Pic^{\md^{\bullet}}(X^{\bullet})$ and notice that    the sections of $L^{\bullet}$ are constant along $E$ (as $\deg_EL^{\bullet}=0$), hence they descend to sections of $L$.
\eqref{mddelta=} is proved.

For \eqref{deltageq} it is enough to show that
for every  $\md\in \Div (G)$ with $[\md]=\delta$ we have
\begin{equation}
\label{mddelta}
\rmax(X,\md)\geq r(X^{\bullet},\delta^{\bullet}).
\end{equation}
Fix such a $\md$; consider $\md^{\bullet}\in \Div(G^{\bullet})$.
Recall that $r(X^{\bullet},\delta^{\bullet})$ is the minimum   of all $\rmax(X^{\bullet}, \md')$ as $\md'$ varies in $\delta^{\bullet}$.
Hence  \eqref{mddelta} follows from  \eqref{mddelta=}.
We have thus proved  \eqref{deltageq}, and, since \eqref{alggeq} follows trivially from it, we are done.
\end{proof}

Let $G$ be a graph admitting some loops and, by abusing  notation, let $G^{\bullet}$ be the loopless graph obtained by inserting
a vertex in the interior of every loop. From the previous result  we derive the following:

\begin{prop}
\label{redloopless}
If $r_{G^{\bullet}}(\delta^{\bullet})=\ralg (G^{\bullet},\delta^{\bullet})$, then  $ r_G(\delta)=\ralg (G,\delta)$.
\end{prop}
\begin{proof}

By iterating the construction described  at the beginning of the subsection, we have
a natural  map
$$
\Pic(G) \to \Pic (G^{\bullet}); \quad \quad \delta \mapsto \delta^{\bullet}
$$
such that $r_G(\delta) =r_{G^{\bullet}}(\delta^{\bullet})$.
By hypothesis  we have
$$
r_G(\delta) =r_{G^{\bullet}}(\delta^{\bullet})=\ralg(G^{\bullet},\delta^{\bullet}) \leq \ralg (G,\delta),
$$
where the last inequality follows from Proposition~\ref{proploopless}. By Theorem~\ref{rankthm} the statement follows.
\end{proof}

\section{When is $r_G(\delta)=  \rgdel$?}
\label{spesec}
The purpose of this section is to find cases   the answer to Problem~\ref{toy2} is affirmative.
Throughout the section we shall restrict our attention to weightless, loopless graphs, unless we specify otherwise.
 \subsection{Special algebraic curves}
Let $G$ be a weightless, loopless graph; we   now look for curves $X\in \MaG$ which are likely to realize
the  inequality $\RX\geq r_{G}(\md)$
for      $\md\in \Div(G)$.
 We shall  explicitly describe   some    such curves,  after  recalling some  notation
 (see \cite{gac}).

Let  $V$, $E$ and $H$ be,  respectively,  the set of  vertices, edges and half-edges of $G$.
We have the following structure maps: the {\it endpoint} map
$$
\epsilon: H\to V;
$$
the {\it gluing} map, which is surjective and two-to-one
$$
\gamma:H\to E.
$$
The gluing map $\gamma$ induces   a fixed-point-free involution on $H$, denoted by $\iota$,   whose orbits, written
 $[h,\ov{h}]$, are identified  with the edges of $G$.

For any $v\in V$, we denote by $H_v=\epsilon^{-1}(v)$ and $E_v=\gamma(\epsilon^{-1}(v))$  the sets of  half-edges and edges
adjacent to $v$.
We denote by
$$
H_{v,w}=H_v\cap \gamma^{-1}(E_w)
$$
(the set of half-edges adjacent to $v$ and glued to a  half-edge adjacent to $w$).
We obviously have $\iota(H_{v,w})=H_{w,v}$.

Let $X$ be a curve  dual to $G$;  we write $
X=\cup_{v\in V}C_v
$ as usual.
We have a set $P_v\subset C_v$  of labeled  distinct points of  $C_v$ mapping to smooth points of $C_v$:
$$
P_v:=\{p_h, \  \forall h\in H_v\} =\sqcup_{w\in V} P_{v,w},
$$
where
$$
P_{v,w}:=\{p_h, \  \forall h\in H_{v,w}\}\subset C_v.
$$
We will use the following explicit description
of $X$
\begin{equation}
\label{Xdec}
X=\frac{\sqcup_{v\in V}C_v}{\{p_h=p_{\ov{h}},\  \forall h\in H\}}.
\end{equation}

\begin{defi}
\label{specdef}
 Let $G$ be a weightless, loopless graph and  $X\in \MaG$.
$X$ is     {\it special}  if
there exists a collection   $$\{\phi_{v,w}: (C_v; P_{v,w})\la (C_w; P_{w,v}), \  \  \forall v,w\in V\},$$ where
$\phi_{v,w}$ is an isomorphism of pointed curves
such that for every $u,v,w\in V$ and $h\in H_{v,w}$ the following properties hold:
\begin{enumerate}[(a)]
\item
$\phi_{v,w}(p_h)=p_{\ov{h}}$;
\item
$\phi_{v,w}^{-1}=\phi_{w,v}$;
\item
$\phi_{v,u}=\phi_{w,u}\circ \phi_{v,w}$.
\end{enumerate}
If $G$ is  not connected, $X\in  \MaG$   is defined to be special if so is every connected component.
\end{defi}
\begin{example} If   $G$ has only vertices of valency at most $3$ then every curve $X\in \MaG$ is special.
 \end{example}
\begin{example}
\label{spex}
We say that  $G$ is a   {\it binary  graph} of genus $g$
if  $G$ consists two vertices joined by $g+1$ edges. If $G$ is a binary graph of genus $g\geq 2$, then
$\dim  \MaG = 2(g-2)$, and the locus of special curves in it has dimension $g-2$.
 \end{example}
\begin{remark}
Let $X$ be a special curve. Then every subcurve of $X$, and every partial normalization of $X$, is special.
Moreover, let $x\in X$ be a nonsingular point of $X$ lying in the irreducible component $C_u$;
then for every component $C_v$ of $X$ the curve
$$
X':=\frac{X}{x=\phi_{u,v}(x)}
$$
is also special. The quotient map $\pi:X\to X'$ describes $X$ as a partial normalization of $X'$.
We say that $X$ is dominates $X'$.
\end{remark}
\begin{lemma}
\label{existspecial}
For every weightless, loopless  graph $G$, the set $\MaG$ contains a special curve.
\end{lemma}
\begin{proof}
The proof is by induction on the number of vertices of $G$; if  $ |V(G)|=1$  there is nothing to prove.

Suppose $|V(G)|\geq 2$. Let $u\in V(G)$ and let $G'=G-u$ be the graph obtained by removing $u$ and all the edges adjacent to it;
we  choose $u$ so that $G'$ is connected (it is well known that such a vertex $u$ exists for any connected graph $G$). Let $C_u$ be a copy of $\PP^1$.
By induction there exists a special curve $X'$ having $G'$ as dual graph;
for every $w,v\in V(G')$ let $\phi'_{w,v}:C_w\to C_v $ be the isomorphisms associated to $X'$.

We now pick $v\in V(G')$ and
  fix an  isomorphism $\phi_{v,u}:C_v\to C_u$.
Now, for any other  vertex $w\in V(G')$ we set
$$\phi_{w,u}:=\phi_{v,u}\circ \phi_{w,v}';$$
  we also set $\phi_{w,v}= \phi'_{w,v}$.
 If
  $H_{u,w}(G)$ is not empty we
pick a set  of distinct points $P_{w,u}\subset C_w\subset X'$ labeled by $H_{w,u}(G)$, such that $P_{w,u}$ does not intersect any $P_{w,w'}$ with $w'\neq u$, 
and such that   $\phi_{w,u}(P_{w,u})$ does not intersect any $ \phi_{w',u}(P_{w',u})$ with $w'\neq w$;
we
  set $P_{u,w}:=\phi_{w,u}(P_{w,u})$.
  Now  let  $X$ be obtained by gluing $C_u$ to $X'$ by identifying $p\in P_{u,w}$ with
  $\phi_{u,w}(p)$ for every $p\in P_{u,w}$ and every $w\in V(G')$.
  It is clear that  $X$ is a  special curve.
  \end{proof}

\subsection{Binary curves} A {\it binary curve} is a curve whose dual graph is binary, as defined in    Example~\ref{spex}.
For such a curve we write $V(G)=\{v_1,v_2\}$ and $X=C_1\cup C_2$, so that $C_i=C_{v_i}$ and $C_i\cong \PP^1$;
recall that $v_1\cdot v_2=g+1$ where $g$ is the genus of $X$.

Let us show that for binary curves the answer to Problem~\ref{toy2} is ``yes".

\begin{prop}
\label{binary}
Let $G$ be
a binary graph   of genus $g$.
Then $r_G(\delta)=\rgdel$ for every $\delta \in \Pic(G)$.
\end{prop}
\begin{proof}
By Theorem~\ref{rankthm} it suffices to prove $\rgdel\geq r_G(\delta)$.
In other words, it suffices to prove that there exists $X\in \MaG$ such that
for every $\md \in \delta$ there exists $L\in \Pic^{\md}(X)$ for which
$r(X,L)\geq r_G(\delta)$.

We can assume $r_G(\delta)\geq 0$ and $0\leq |\delta|\leq 2g-2$, see \cite[Theorem  2.9 and Lemma 2.4]{Cjoe}.

 Since $r_G(\delta)\geq 0$ we can choose $\md=(a,b)\in \delta$ such that $0\leq a\leq b$.
There are two cases.

$\bullet$ Case 1: $b\leq g$.

Then  $r_G(\md)=a$, by Proposition~\ref{rank}.
Let $X=C_1\cup C_2$ be a   special binary curve. Then there clearly exists $L\in \Pic^{(a,a)}(X)$ such that
$r(X,L)=a$, hence $\rmax(X, (a,a))\geq a$; Remark~\ref{monotone} yields
$\RX\geq \rmax(X, (a,a))\geq a$.

Now let $\md' \in \delta$ be a different representative, so that
$\md'=(a-n(g+1),b+n(g+1))$ for $n\in \Z$ with $n\neq 0$.
Then for any $L'\in \Pic^{\md'}(X)$,
using the simple estimate
$$
r(X,L')\geq h^0(C_1,L'_{|C_1})+ h^0(C_2,L'_{|C_2})-1-(g+1)
$$
 and recalling that $a\leq b\leq g$,
 we easily get
$$
r(X,L')\geq a+|n|(g+1)-(g+1)\geq a.
$$
So we are done.

$\bullet$ Case 2: $b\geq g+1$.
We claim that
\begin{equation}
\label{RRab}
r_G(\md)=a+b-g.
\end{equation}
In fact from  Proposition~\ref{rank} and the fact that  $b\geq g+1$ we obtain
$$
r_G(K_G-(a,b))=r_G(g-1-a,g-1-b)=-1.
$$
Therefore \eqref{RRab} follows from
  Riemann-Roch.

   On the other hand pick any $X\in \MaG$ and any representative
   $\md'=(d_1,d_2)\in \delta$, so that  $d_1+d_2=a+b$.
Let $L\in \Pic^{\md'}(X)$.

Denote by $X^{\nu}$ the normalization of $X$ and by $L^{\nu}$ the pull back of $L$ to it.
If $d_i\geq 0$ for $i=1,2$ we have
$$
r(X,L)\geq r(X^{\nu},L^{\nu}) -(g+1)=d_1+d_2+1-g-1=a+b-g=r_G(\md).
$$
If $d_1\leq -1$ then $d_2\geq 1+a+b$ and we have
$$
r(X,L)\geq r(X^{\nu},L^{\nu}) -(g+1)=d_2-g-1\geq a+b-g=r_G(\md),
$$
so we are done.
\end{proof}

\begin{remark}
The proof gives a slightly stronger statement. Indeed   in case 1, we proved that for every special curve $X\in \MaG$ we have $\rdel= r_G(\delta)$.
In case 2
we have
$r(X,L)\geq r_G(\delta)$
 for every $X \in \MaG$, every $\md\in \delta$ and every $L\in \Pic^{\md}(X)$.
\end{remark}

\subsection{The rank-explicit case}
For  a weightless loopless graph $G$, in Proposition~\ref{rank}   we proved  that if a divisor $\md$ is reduced with respect to a vertex $u$ of minimal degree,
then $r_G(\md)=\md(u)$, unless $\md(u)<-1$ in which case $r_G(\md)=-1$. We think of such a divisor  as
 ``explicitly exhibiting"  its rank (equal to its minimal entry); this motivates the  terminology below.
\begin{defi}
\label{rex} 
Let $G$ be a weightless  loopless
graph.
We say that a divisor $\md\in \Div(G)$ is {\emph {rank-explicit}}     if
$\md$ is  $u$-reduced for some   vertex  $u$  such that $\md(u)=\agd$.

More generally, we say that a divisor $\md$ on any graph   $G$   is {\emph {rank-explicit}}     if
$\md$ is  $u$-reduced for some   vertex  $u$  such that $\mdr(u)=\agd$.

We say that a divisor class $\delta\in \Pic(G)$ is {\emph {rank-explicit}} if it admits a rank-explicit representative.
\end{defi}

\begin{remark}
\label{rexrank}
 If $\md$ is rank-explicit, then $r_G(\md)=\agd$, by Prop.~\ref{wrank}.
\end{remark}

\begin{remark}
\label{exrex}
 By Lemma~\ref{E:reduced}, if $r_G(\delta)\leq 0$ then $\delta$ is rank-explicit.
\end{remark}
 \begin{example}
\label{notrex}
Not all divisor classes are rank-explicit. For example, on a binary graph of genus 1 the divisor class $\delta=[(0,2)]$
 has rank 1 and is not rank-explicit.
 \end{example}

By Lemma~\ref{P:sections}, if $\md$ is a $u$-reduced divisor on a graph $G$, then for every $X\in \MaG$ and every $L\in \Pic^{\md}(X)$,
  the restriction map
 $H^0(X,L)\to H^0(C_u,L_{C_u})$   is  injective  (where $C_u\subset X$ is the component corresponding to $u$).
 The following lemma tells us under which conditions the restriction map is an isomorphism, using
 rank-explicit divisors and special curves.

 \begin{lemma} \label{explicitLemma} Let $\md$ be a   rank-explicit divisor on a weightless loopless graph $G$.
Then, for every special curve $X\in\MaG$, there exists a line bundle $L\in\Pic^{\md}(X)$, such that $H^0(X,L) \cong H^0(\PP^1,\O(\agd))$.
 \end{lemma}
 \begin{proof}

Set $\ell=\agd$; let $X\in\MaG$ be a special curve.
By Remark \ref{monotone}, we may assume that $\md=(\ell,\ell,\ldots, \ell)$. If $\ell=-1$ there is nothing to prove, so we will  assume $\ell\geq 0$.

We write $X=\cup C_i$, and fix isomorphisms of pointed curves
$\phi_{i,j}:(C_i,P_{i,j})\to (C_j,P_{j,i})$ as in Definition~\ref{specdef}.
Recall that if $x\in X$ is a node whose branches are $x^i\in C_i$ and $x^j\in C_j$,
then $\phi_{i,j}(x^i)=x^j$.

For $L\in \Pic ^{\md}(X)$,  we write $L_i:=L_{C_i}$  for every $i$, and fix an isomorphism  $L_i\cong \O_{\PP^1}(\ell)$;  now, each $\phi_{i,j}$ induces an isomorphism
\begin{equation}
\label{chi}
\chi_{i,j}:H^0(C_i,L_i)\la H^0(C_j,L_j);\quad \quad s\mapsto s\circ \phi_{j,i}.
\end{equation}
By hypothesis, $\md$ is $u$-reduced for some $u\in V$; consider the corresponding    Dhar decomposition  $$
V=Y_{0}\sqcup Y_1\ldots\sqcup Y_{l},
$$
(see subsection~\ref{S:Dhar}).  We can fix, with no loss of generality, an ordering $V=\{v_0=u, v_1, \ldots, v_t  \} $
 compatible with the decomposition  (i.e.   if $v_i\in Y_h$ and $v_{i'}\in Y_{h'}$ with $h<h'$, then $i<i'$),
 and such that   it   induces a filtered sequence of connected subcurves
 $$
C_0=Z_0\subset   \ldots \subset Z_m\subset \ldots \subset Z_t=X
$$
with $Z_m= \cup_{i=0}^mC_{i}$ and $C_i=C_{v_i}$.

We pick $L\in \Pic^{\md}(X)$
such that the gluing constants over the nodes are all equal to $1$;
we will prove that for every $m= 0, \ldots, t$,  there is an isomorphism
\begin{equation}
\label{Hm1}
H^0(Z_m, L_{Z_m})\cong H^0( C_0,L_{C_0}),
\end{equation}
which, as $H^0(C_0,L_{C_0}) \cong H^0(\PP^1,\O(\ell))$, implies the Lemma.

The proof of \eqref{Hm1} will go by induction on $m$. The case
 $m=0$ is obvious.
Assume the statement holds for $ m-1$. Write
$
Z_{m}=Z_{m-1}\cup C_{m}
$; for every node $x_{\alpha}\in  Z_{m-1}\cap C_{m}$,
we denote its branches by $x_{\alpha}^m\in C_m$ and $x_{\alpha}^{i(\alpha)}\in C_{i(\alpha)}\subset Z_{m-1}$.
Consider the exact sequence
$$
0\to H^0(Z_{m},L_{Z_{m}})\stackrel{\rho}{\to} H^0(Z_{m-1},L_{Z_{m-1}})\oplus
H^0(C_{m},L_{m})
 \stackrel{\pi}{\la} \bigoplus_{\alpha}k(x_{\alpha}),
$$
where $\rho$ is the  restriction map
(i.e. $\rho(s)=s_{|Z_{m-1}}\oplus s_{|C_{m}}$), $k(x_{\alpha})$ is the skyscraper sheaf
supported on $x_{\alpha}$ (equal to $k$ on $x_{\alpha}$),   the point $x_{\alpha}$ ranges in $Z_{m-1}\cap C_{m}$,
 and
$$
\pi(s\oplus t):=\oplus_{x_{\alpha}\in C_m\cap Z_{m-1}}(s(x_{\alpha}^{i(\alpha)})-t(x_{\alpha}^m)).
$$

By our ordering of the vertices of $G$, and
since $\md$ is $u$-reduced,
we have $|C_m\cap Z_{m-1}|>\deg_{C_m}L$, so  the restriction of $\pi$ to $ H^0(C_{m},L_{m})$ is injective (a section of $L_{m}$ cannot have more zeroes than its degree).

We now claim that $\pi$ induces an isomorphism between $H^0(C_{m},L_{m})$ and
$ \im \pi. $
It suffices to prove that for any $s\in H^0(Z_{m-1},L_{Z_{m-1}})$, we have $\pi(s)\in \pi (H^0(C_{m},L_{m}))$. 
Recall   that $x_{\alpha}^{i(\alpha)}=\phi_{m, i(\alpha)}(x_{\alpha}^m)$, hence

$$
\pi(s)=\sum _{\alpha}s(x_{\alpha}^{i(\alpha)})=\sum _{\alpha}s(\phi_{m, i(\alpha)}(x_{\alpha}^m))=
\sum _{\alpha}\chi_{ i(\alpha),m}(s_{|C_{ i(\alpha)}})(x_{\alpha}^m)
$$
by \eqref{chi}.
Since $\chi_{i(\alpha),m}(s_{|C_{ i(\alpha)}})\in H^0(C_{m},L_{m})$
the proof of the claim is complete.
 Therefore
 $\im \pi \cong H^0(C_{m},L_{m})$,
  hence
  $$H^0(Z_m,L_{Z_m})\cong H^0(Z_{m-1},L_{Z_{m-1}})\cong H^0(C_u,L_u)$$ by induction;
 \eqref{Hm1} is proved, and the Lemma with it.
 \end{proof}

\begin{thm}
\label{rexthm}
 Let $G$ be    weightless and  loopless;   let $\delta\in \Pic (G)$ be rank-explicit.
 Then $\rgdel=r_G(\delta)$.
\end{thm}
Set $r=r_G(\delta)$;
by Theorem~\ref{rankthm} we can assume $r\geq 0$.
Let $\md$ be a rank-explicit representative for $\delta$;
then $r_G(\delta) = \agd$ by Remark~\ref{rexrank}. Therefore  Theorem \ref{rexthm}
 is a special case of the following more precise result.

\begin{prop}
\label{RRthm}
Let $G$ be  weightless and loopless; pick
  $\md\in \Div(G)$
and set
$\delta=[\md]$. Let  $X\in \MaG$ be a special curve. Then
\begin{enumerate}[(a)]
\item
\label{RR1}
$\rmax(X,\md)\geq \agd$.
\item
\label{RR2}
$r(X,\delta)\geq\agd$.
\item
\label{RR3}
If $r_G(\md)=\agd$ then $\rgdel=r_G(\delta)$.
\end{enumerate}
\end{prop}

\begin{proof}
Set $\ell=\agd$.
Let $X\in\MaG$ be a special curve. If $\md$ is rank explicit, then part \eqref{RR1} is an immediate consequence of Lemma \ref{explicitLemma}.

To treat the general case, fix a vertex $u$ such that $\md(u)=\ell$, and suppose $\md$ is not reduced with respect to $u$. We choose a $G$-saturation, $G'$, with respect to $\md$ and $u$ (see Definition~\ref{Dhardef});
notice that $\md$ is rank-explicit on $G'$.
Now,
 we  construct a   special curve, $X'$, having $G'$ as dual graph, by gluing together some   points in  $C_u$ to points on other components of $X$; hence $X$ dominates $X'$ by a birational map
$$
\pi:X\la X'.
$$
By Lemma \ref{explicitLemma}, there exists $L'\in \Pic ^{\md}(X')$ satisfying
$r(X',L')= \ell$.
Consider $L= \pi^*L'$; then $L\in  \Pic ^{\md}(X)$, and, of course,
$$
r(X,L)\geq r(X',L')\geq \ell,
$$
as  required.

For part \eqref{RR2} we must prove that for every $\mc\sim  \md$ with $\md\neq \mc$ there exists $M\in \Pic^{\mc}(X)$
such that $r(X,M)\geq \ell$.
We have
$ 
\mc+\mt=\md,
$ 
for some   $\mt\in \Prin(G)$. We argue as for Corollary~\ref{cor0-1}, with same notation (repeating some things for convenience).
 By Remark~\ref{tZrk}   we have
\begin{equation}
 \label{mtZ}
\mt_{|Z}\leq (\mt_{Z})_{|Z},
\end{equation}
with    $Z\subsetneq V(G)$.  We denote by $Z$ and $Z^c$ the subcurves of $X$   corresponding to the vertices in $Z$
and $Z^c$. By Lemma \ref{explicitLemma} applied to the special curve $Z$, 
or to a connected component of $Z$ if $Z$ is not connected, 
there exists a line bundle $L_Z\in \Pic^{\md_{|Z}}(Z)$ satisfying
\begin{equation}
 \label{mZZl}
 r(Z,L_Z)\geq \ell.
 \end{equation}
Let
  $Z\cdot Z^c\in \Div(Z)$ be the    divisor cut on $Z$ by $Z^c$;
 we have
\begin{equation}
 \label{mZZ}
  \mdeg_ZZ\cdot Z^c= (\mt_{Z^c})_{|Z}=-(\mt_{Z})_{|Z}.
\end{equation}
  Now,   for every  $M\in \Pic ^{\mc}(X)$  we have
\begin{equation}
 \label{MMcc}
r(X,M)\geq r(Z, M_Z(-Z\cdot Z^c)).
\end{equation}
Moreover, using \eqref{mZZ} and \eqref{mtZ}
$$
\mdeg _ZM(-Z\cdot Z^c)=\mc_{|Z}-\mdeg_ZZ\cdot Z^c= \mc_{|Z} + (\mt_{Z})_{|Z}\geq \mc_{|Z}+\mt_{|Z}=\md_{|Z}.
$$
We can therefore  pick $M\in \Pic ^{\mc}(X)$   such that its restriction to $Z$ satisfies
$M_Z=L_Z(Z\cdot Z^c+E)$ for some effective divisor  $E$  on $Z$. By \eqref{MMcc} and \eqref{mZZl} we have
$$
r(X,M)\geq r(Z,L_Z(Z\cdot Z^c+E-Z\cdot Z^c)) \geq r(Z,L_Z)\geq \ell,$$
and part~\eqref{RR2} is proved.
Part~\eqref{RR3} follows from \eqref{RR2} and Theorem~\ref{rankthm}.
\end{proof}

By the following Example~\ref{wbinex} we have that neither Theorem~\ref{rexthm}, nor Proposition~\ref{binary},
 extend to weighted graphs.

\begin{example}
\label{wbinex}
Let $G$ be a   graph with two vertices $v_1$ and $v_2$, whose weights are $\w(v_1)=1$, $\w(v_2)=2$,
and such that $v_1\cdot v_2>12$. Consider the divisor  $\md =(3,4)$.
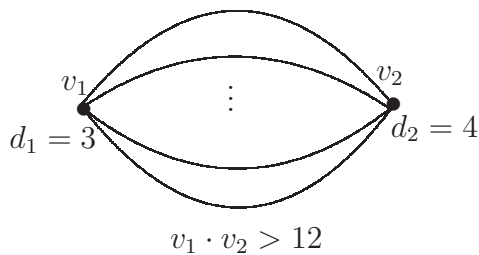
\begin{center}
\unitlength .9mm 
\linethickness{0.4pt}
\ifx\plotpoint\undefined\newsavebox{\plotpoint}\fi  
\begin{picture}(64.781,38.625)(10,62)
\qbezier(19,85.5)(42,113.625)(65,85.25)
\qbezier(64,85.25)(41.375,100.25)(19.25,85.25)
\qbezier(19.25,85.25)(42.5,55.625)(64.75,85.5)
\qbezier(64.75,85.5)(42.375,67.25)(19.5,85)
\put(19.5,85){\circle*{2}}
\put(64.75,85.75){\circle*{2.062}}
\put(41,87.75){\makebox(0,0)[cc]{$\vdots$}}
\put(18.25,88.5){\makebox(0,0)[cc]{$v_1$}}
\put(15,80.5){\makebox(0,0)[cc]{$d_1=3$}}
\put(64.25,90){\makebox(0,0)[cc]{$v_2$}}
\put(70.75,82.25){\makebox(0,0)[cc]{$d_2=4$}}
\put(43.25,65.75){\makebox(0,0)[cc]{$v_1\cdot v_2>12$}}
\end{picture}
\end{center}

We claim that  every
curve $X\in \MaG$ satisfies $\rmax(X,\underline{d})<r_{G}(\underline{d})$.
Notice that $\mdr=(2,2)$ and $\md$ is reduced with respect to both vertices.
So by Proposition~\ref{wrank}
we have  $r_{G}(\underline{d})=2$, and $\md$ is rank-explicit.

Now let
$L$ be a line bundle on $X$ with $\underline{\deg} \ L =\md$, where we write $X=C_1\cup C_2$ and $L_i=L_{|C_i}$, with $C_i$ corresponding to $v_i$.
Let
us see that $r(X,L)< 2$. Since this is
independent of the choice of $X$ and $L$, the claim will follow.

Assume by contradiction that   $r(X,L)= 2$. Since
$\underline{\deg} \ L =(3,4)$, we have  that
$r(C_1,L_{1})=r(C_2,L_{2})=2$. Note that if a section of $L_{1}$ or $L_{2}$
can be extended to all of $X$, then the extension is unique, since the number $|C_1\cap C_2|$ is large. Hence, the map $\phi_{L}:X\to\mathbb{P}^{2}$ determined
by $L$ restricts to non-degenerate maps $\phi_{1}:C_{1}\to\mathbb{P}^{2}$ and $\phi_{2}:C_{2}\to\mathbb{P}^{2}$.

The image of $\phi_{1}$ is an irreducible curve of degree $3$, so it is either a cubic or a line with multiplicity
3; since $\phi_{1}$ is non-degenerate, the image is a cubic. Similarly,
the image of $\phi_{2}$ is a non-degenerate irreducible curve of degree $4$, so it is either a (singular) quartic, or a conic
of multiplicity $2$.

Hence, $\phi_L(X)$ consists of two distinct irreducible curves of degrees
$3,4$. By B\'ezout Theorem, they intersect in at most $12$ points,
which is a contradiction.
\end{example}

The next example shows that the hypothesis that $\delta$ be rank-explicit is really needed
in Theorem~\ref{rexthm}.

 \begin{example}
 \label{counterex1} Let $X=C_1\cup C_2 \cup C_3$ be a curve with $3$ rational components meeting as follow (see the picture below).
$$C_1\cdot C_3=3\mbox{ and }C_2\cdot C_3 >6.$$

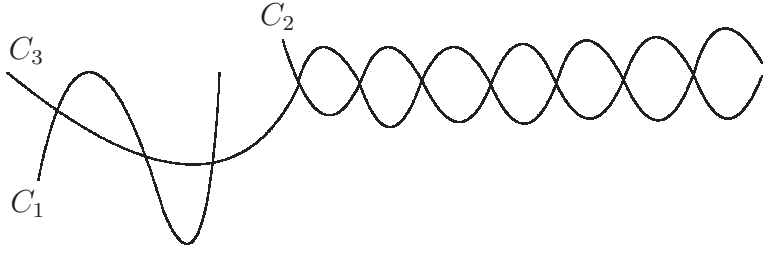
\begin{center}
\unitlength .9mm 
\linethickness{0.4pt}
\ifx\plotpoint\undefined\newsavebox{\plotpoint}\fi  
\begin{picture}(111.5,35.75)(0,80)
\qbezier(41.5,111.5)(47.25,92.875)(53,105.75)
\qbezier(53,105.75)(56.5,115.5)(62,105.25)
\qbezier(62,105.25)(67.5,93.75)(72,105.25)
\qbezier(72,105.25)(76.375,116.625)(81.25,105.5)
\qbezier(81.25,105.5)(86.625,94.125)(91.5,106.25)
\qbezier(91.5,106.25)(95.5,117.625)(101.5,106.5)
\qbezier(101.5,106.5)(106.5,93.625)(111.5,106.25)
\qbezier(111.5,108.25)(105,119.125)(101.5,106.5)
\qbezier(101.5,106.5)(96.625,93.375)(91.25,105.75)
\qbezier(91.25,105.75)(85.125,117.5)(81.5,105.25)
\qbezier(81.5,105.25)(76.875,93.25)(71.75,105.25)
\qbezier(71.75,105.25)(66.25,115.625)(61.75,105.5)
\qbezier(61.75,105.5)(57.25,92.25)(52.75,105)
\qbezier(52.75,105.25)(46.75,115.625)(43.75,105.5)
\qbezier(43.75,104.75)(32.5,80.75)(1.25,106.75)
\qbezier(5.75,91)(12.625,124.75)(24,86.5)
\qbezier(24,86.5)(30.375,70.375)(32.25,106.75)
\put(4.25,87.5){\makebox(0,0)[cc]{$C_1$}}
\put(4,110){\makebox(0,0)[cc]{$C_3$}}
\put(40.75,115){\makebox(0,0)[cc]{$C_2$}}
\end{picture}
\end{center}
Let $G$ be the dual graph of $X$ and let $\md= (1,2,3)$ be the  divisor in $\Div (G)$ such that the degree of $\md$ on $C_i$ is $i$.
We claim that
$$
2=r_G(\md)>\ralg(G,\md).
$$
It is easy to see that $r_G(\md)=2$.
To prove the claim
we will prove that
for every $L\in \Pic^{\md}(X)$ we have $r(X,L)<2$.

Assume, by contradiction, that $r(X,L)=2$. Then $L$ defines a non-degenerate map $\phi :X\to \mathbb P^2$.
We will treat all possible cases, getting a contradiction in each of them.

Case 0: $\phi(C_3)$ is point. But then $\phi(X)$ is a point (for otherwise $\phi(C_1)$ or $\phi (C_2)$ would have a singular point of too high multiplicity), which is not possible as $\phi$ is  non-degenerate.

Case 1: $\phi(C_3)$ is a line. Then one sees easily that $\phi(C_2)=\phi(C_3)$. Hence $\phi(C_1)$ must be a different  line  (for   the map $\phi$ is non-degenerate). But then the restriction of $\phi$ to $C_1$ is an isomorphism, so $\phi$ maps the three points of $C_1\cap C_3$ in different points, which is impossible as $\phi(C_3)$ is a line other than $\phi(C_1)$.

Case 2: $\phi(C_3)$ is a conic. Hence $\phi$ maps $C_3$ isomorphically to its image and  $L$ has a base point $p\in C_3$.
We claim that $\phi(C_2)=\phi(C_3)$. Indeed,
 $\phi(C_2)$ cannot be a point (for it would be a singular  point of $\phi(C_3)$);  hence $\phi(C_2)$ is a curve of degree at most 2, which, if different from  $\phi(C_3)$, would meet  $\phi(C_3)$ in degree greater than 4, contradicting B\'ezout  Theorem.

 Let us now consider $C_1$; if the base point $p$ is one of the three points
where $C_3$ meets $C_1$, then $L$ has a base point  on $C_1$ and hence $\phi(C_1)$ is a point. This is impossible, since $C_1$ meets $C_3$ in two more points which, as we said, have distinct images via $\phi$. If $p\not\in C_1$, arguing as before we get that $\phi(C_1)$ cannot be a point, and hence it is a line, which intersects $\phi(C_3)$ in three points. By B\'ezout Theorem this is impossible.

Case 3: $\phi(C_3)$ is a cubic. Then $\phi$ maps $C_3$ birationally onto its image, and $\phi(C_2)$ cannot be a point (for it would be a   point of too high multiplicity on a cubic).
Hence $\phi(C_2)$ is a curve of degree at most 2, which meets $\phi(C_3)$ in degree greater than 6. This contradicts B\'ezout  Theorem.

\end{example}
\begin{remark}
By slightly modifying the example above, we have an example where   strict inequality between the algebraic and combinatorial rank holds  for a
{\it  simple graph} (a graph with at most one edge between any two  vertices).
Indeed, keeping the notation of the  example, consider the graph $G'$ obtained from $G$ by adding a vertex in the interior of each edge, and let
$\md'\in\text{Div}(G')$ be the divisor obtained by extending $\md$ by zero. Then by \cite[Theorem 1.4]{bakersp},
$ 
r_{G'}(\md') = r_{G}(\md) = 2. 
$ 

Let $X'$ be any curve  dual to $G'$, let $L'$ be a line bundle on $X'$ such that $\mdeg  L' = \md'$, and let $u\in V(G')\smallsetminus V(G)$. Since $\md'(u)=0$, the map $\phi':X'\to \PP^2$ corresponding to $L'$ maps the  component $C_u$ to a point.
Therefore  the number of intersection points of $\phi'(C_1), \phi'(C_2)$, and $\phi'(C_3)$ is the same as for $\phi$, and the exact argument as in the example above yields
$
\ralg(G',\md') < 2 = r_{G'}(\md').
$
\end{remark}

We know that if the combinatorial rank is $-1$ or $0$, then the combinatorial and algebraic rank are equal.
Notice that in all our  examples  where we do not have equality,   the combinatorial rank is  equal
to 2.
So the following case is  open:
\begin{question}
Suppose $r_G(\delta)=1$. Is $r_G(\delta)=\ralg(G,\delta)$?
\end{question}

Finally, we may consider the following
 variant of Problem~\ref{toy2}:
\begin{question}\label{Q:weakConj}
Let $G$ be a graph and $\delta\in \Pic(G)$. Does there exist a representative $\md_0\in\delta$ and $X\in \MaG$ such that
$$\rmax(X,\md_0)=r_G(\delta)?$$
\end{question}
Notice that the answer to question \ref{Q:weakConj} is   positive in both our counterexamples.  For Example \ref{wbinex} we take $\md_0=(3,4)+\mt_{v_1}=(3-k,4+k)$, where $k:=v_1\cdot v_2>12$;  for Example \ref{counterex1} we can take  $\md_0=(1,2,3)+\mt_{v_1}=(-2,6,2)$.

\
 
\noindent
 Lucia Caporaso - caporaso@mat.uniroma3.it

 \noindent
 Dipartimento di Matematica e Fisica, Universit\`a Roma Tre  
 
  \noindent
  Largo San L. Murialdo 1 - 00146 Rome (Italy)

 \
 
 \noindent
 Yoav Len - yoav.len@yale.edu 
 
 \noindent
 Mathematics Department, Yale University
 
  \noindent
10 Hillhouse Ave, New Haven, CT, 06511 (USA)

 \

\noindent
 Margarida Melo -  melo@mat.uniroma3.it
 
  \noindent
 Dipartimento di Matematica e Fisica, Universit\`a Roma Tre  
 
  \noindent
  Largo San L. Murialdo 1 - 00146 Rome (Italy) 
 
and 

 \noindent
 CMUC and Mathematics Department of the University of Coimbra
 
  \noindent Apartado 3008, EC Santa Cruz
3001 -  501 Coimbra (Portugal)


\begin{thebibliography}{EGKH02}
 \bibitem[AB]{AB}
Amini, O.; Baker, M.: {\it Linear series on metrized complexes of algebraic curves.} To appear in Math. Ann..



 \bibitem[AC]
 {AC} Amini, O.; Caporaso,  L.: {\it Riemann-Roch theory for weighted graphs and tropical curves.}
Adv. in Math. 240  1--23 (2013).

   
\bibitem[GAC]{gac} Arbarello, E.;   Cornalba, M.;  Griffiths, P. A.:
{\it Geometry of algebraic curves. Vol. II.} With a contribution by  Harris J. D.
Grundlehren der Mathematischen Wissenschaften 268,
Springer-Verlag (2011).

 \bibitem[BdlHN]{BdlHN} Bacher, R.; de la Harpe, P.; Nagnibeda, T.:
 {\it The lattice of integral flows and the lattice of integral cuts
 on a finite graph.}  Bull. Soc. Math. France (1997), no. 125 , 167--198.

\bibitem[B]{bakersp} Baker, M.: {\it Specialization of linear systems from curves to graphs.} Algebra Number Theory 2 (2008), no. 6, 613--653.

\bibitem[BN]{BNRR}Baker, M.; Norine, S.: {\it Riemann-Roch and Abel-Jacobi theory on a finite graph.} Adv. Math.  215  (2007),  no. 2, 766--788.

 
 \bibitem[CH]{CH} Caporaso  L.;  Harris J.: {\it Counting plane curves of any genus.} Invent. Math. 131 (1998), no. (2), 345--392.

 
 
 \bibitem[C1] {Ccliff} Caporaso, L.:  {\it Linear series on semistable curves.} Int. Math. Res. Not. (2011), no. 13, 2921--2969.
 
 

 \bibitem[C2] {Cjoe} Caporaso, L.:  {\it Rank of divisors on graphs: an algebro-geometric analysis.}   A Celebration of Algebraic Geometry - Volume in honor of Joe Harris. 45--65 AMS - Clay Mathematics Institute.
 
  \bibitem[Ca]{Ca} Cartwright, D.:  {\it Lifting rank-2 tropical divisors.} Preprint.
  
\bibitem[CDPR]{CDPR} Cools, F.;  Draisma, J;   Payne, S.;  Robeva, E.: {\it A tropical proof of the Brill-Noether Theorem.} Adv. Math. 230 (2012), no. 2,    759--776.

 \bibitem[DM]{DM}Deligne, P.; Mumford, D.:
 \newblock \emph{ The irreducibility of the space of curves of given genus.}
 \newblock Inst. Hautes \'Etudes Sci. Publ. Math. No 36 (1969) 75--120.

\bibitem[GH]{GH}  
Griffiths, P;  Harris, J.: {\it On the variety of special linear systems on a
general algebraic curve.}  Duke Math. J.  47  (1980), no. 1, 233--272. 

 
 \bibitem[HM]{HMo} Harris, J.; Morrison, I.: {\it Moduli of curves.} Graduate Texts in Mathematics, 187. Springer-Verlag, New York, 1998.
 
 \bibitem[KY1]{KY1} Kawaguchi, S.;   Yamaki, K.:  {\it Ranks of divisors on hyperelliptic curves and graphs under specialization.} Preprint available at arXiv 1304.6979.
  
 \bibitem[KY2]{KY2} Kawaguchi, S.;   Yamaki, K.:  {\it Algebraic rank on hyperelliptic graphs and graphs of genus 3.} Preprint available at arXiv:1401.3935
\bibitem[KM]{KM}Kontsevich, M.;  Manin Y.:  {\it Gromov-Witten classes, quantum cohomology and enumerative geometry.} Comm. Math. Phys. 164 (1994)  525--562.

 

\bibitem[Le1]{Yoav} Len,  Y.: {\it The Brill-Noether rank of a tropical curve.} J. Algebraic Combin. 40 (2014), no. 3, 841--860. 
\bibitem[Le2]{Yoav2} Len, Y.: {\it  A note on algebraic rank, matroids, and metrized complexes.} Preprint available at arXiv 1410.8156.
\bibitem[Lu]{Luo}  Luo, Y.: {\it Rank-determining sets of metric graphs.} J. Combin. Theory Ser. A 118 (2011), No. 6, 1775--1793.


\end{thebibliography}
\end{document}